\documentclass[11pt]{amsart}
\usepackage[top=1in, bottom=1in, left=1in, right=1in]{geometry}

\usepackage{mathtools,mathrsfs,amsthm,amsfonts,amssymb,graphicx,appendix,cancel}
\usepackage{xspace,enumerate,comment,bbm,nth}

\usepackage{tikz,pgfplots}

\usepackage{hyperref}
\usepackage[normalem]{ulem}   

\usepackage[active]{srcltx}

\newcommand{\ol}[1]{\overline{#1}}

\newcommand{\F}{\mathcal F}
\newcommand{\I}{\mathcal I}
\newcommand{\R}{\mathbb{R}}

\newcommand{\Z}{\mathbb{Z}}
\newcommand{\N}{\mathbb{N}}

\renewcommand{\P}{\mathbb{P}}
\newcommand{\PP}{\mathbb{P}}

\newcommand{\Sph}{\mathbb{S}}

\newcommand{\HV}{{\overline H}}
\newcommand{\M}{\R^d}
\newcommand{\TM}{\M\times\M}

\newcommand{\cyl}{(0,+\infty)\times\M}
\newcommand{\ccyl}{[0,+\infty)\times\M}

\newcommand{\Tcyl}{(0,T)\times \M}
\newcommand{\cTcyl}{[0,T)\times\M}
\newcommand{\ccTcyl}{[0,T]\times\M}
\newcommand{\coTcyl}{[0,T)\times\M}

\newcommand{\eps}{\varepsilon}
\renewcommand{\epsilon}{\varepsilon}

\renewcommand{\le}{\leqslant}
\renewcommand{\leq}{\leqslant}
\renewcommand{\ge}{\geqslant}
\renewcommand{\geq}{\geqslant}

\newcommand{\CC}{\D C}
\newcommand{\Ham}{\mathscr{H}}
\newcommand{\Hamdg}{\mathscr{H}_{dG}}
\newcommand{\Lip}{\D{Lip}}
\newcommand{\BUC}{\D{BUC}}
\newcommand{\UC}{\D{UC}}
\renewcommand{\AA}{\mathbb{A}}
\newcommand{\BB}{\mathbb{B}}
\newcommand{\hh}{{h}}
\newcommand{\indep}{\perp\!\!\!\!\perp} 
\newcommand{\D}[1]{\mbox{\rm #1}}

\newcommand{\ucv}{\rightrightarrows}

\newcommand{\bluee}{}
\newcommand{\redd}{}

\newcommand{\EE}{\mathbb{E}}

\newcommand{\RR}{\mathbb{R}}
\newcommand{\NN}{\mathbb{N}}

\newcommand{\m}{\mathbb}

\newcommand{\cond}{\;\middle\vert\;}

\usepackage{cleveref}

\theoremstyle{plain} 
\newtheorem{theorem}{\sc Theorem}[section]

\newtheorem{cor}[theorem]{\sc Corollary}
\newtheorem{lemma}[theorem]{\sc Lemma}
\newtheorem{proposition}[theorem]{\sc Proposition}
\newtheorem{prop}[theorem]{\sc Proposition}

\theoremstyle{definition} 
\newtheorem{definition}[theorem]{\sc Definition}

\theoremstyle{remark} 
\newtheorem{remark}[theorem]{\sc Remark}

\def\AND{\text{ and }}

\def\FORALL{\text{ for all }}

\begin{document}

%

\title[Stochastic homogenization of HJ equations:  
a differential game approach]{Stochastic homogenization of HJ equations:\\ 
a differential game approach}
\author{Andrea Davini \and Raimundo Saona \and Bruno Ziliotto}
\date{\today}

\maketitle

\begin{abstract}
		We prove stochastic homogenization for a class of non-convex and non-coercive first-order Hamilton-Jacobi equations in a finite-range-dependence environment {\redd for Hamiltonians that can be expressed by a max-min formula. Exploiting the representation of solutions as value functions of differential games, we develop a game-theoretic approach to homogenization. We furthermore extend this result to a class of Lipschitz Hamiltonians that need not admit a global max-min representation. Our methods allow us to get a quantitative convergence rate for solutions with linear initial data toward the corresponding ones of the effective limit problem. 
	
}	
\end{abstract}

\section{Introduction}

In this paper we study the asymptotic behavior, as $\eps \to 0^+$, of solutions to a stochastic Hamilton--Jacobi (HJ) equation of the form
\begin{equation}
\label{eq eps hj}
\tag{HJ$_\eps$}
\partial_t u^\eps + H\!\left(\frac{x}{\eps}, D_x u^\eps, \omega\right)=0
\quad \text{in } \Tcyl,
\end{equation}
for each fixed $T>0$, where $H \colon \TM \times \Omega \to \R$ is a Lipschitz Hamiltonian admitting a max-min representation. The dependence on the random environment $(\Omega, \F, \P)$ enters through the Hamiltonian $H(x,p,\omega)$, whose law is assumed to be \emph{stationary}, i.e., invariant under spatial translations, and \emph{ergodic}, i.e., any translation-invariant event has probability either $0$ or $1$.
{\redd
Under the additional assumptions that the random variables $\big(H(\cdot,p,\cdot)\big)_{p \in \M}$ satisfy a finite-range dependence condition and that the underlying dynamics is oriented, we prove homogenization for \eqref{eq eps hj} (Theorem~\ref{thm:genhom}) and obtain a convergence rate, for solutions with linear initial data, toward the corresponding solutions of the effective limit problem (Theorem~\ref{Result: Concentration property}). This latter, stronger result is stable under local uniform convergence of suitable sequences of Hamiltonians of the above type. As a consequence, homogenization extends to the limiting Hamiltonians (Corollary~\ref{cor:quantitative estimate}), which in general cannot be expressed in max-min form.

A second extension in this direction is provided by Theorem~\ref{thm:genhom2}, where we prove analogous results for a class of Lipschitz Hamiltonians that need not admit a global max-min representation. Using an argument from~\cite[Section~5]{EvSo84}, these Hamiltonians can, however, be written in max-min form \emph{locally in $p$}; this suffices for our proof strategy, which is tailored to this extension. The full set of assumptions and the precise statements of our homogenization results are presented in Section~\ref{sec:assumptions}. We emphasize that the Hamiltonians considered here are noncoercive and nonconvex in $p$.
\smallskip
}

The coercivity of $H$ in the momentum is a condition often assumed in the homogenization theory of first-order HJ equations.
Its role is to provide uniform $L^\infty$ bounds on the derivatives of solutions to \eqref{eq eps hj} and to an associated “cell” problem.
The first homogenization results for equations of the form \eqref{eq eps hj} with coercive Hamiltonians were established in the periodic setting in the pioneering work~\cite{LPV} and later extended to the almost periodic case in~\cite{I_almostperiodic}.
The generalization of these results to the stationary ergodic setting was obtained in \cite{Sou99, RT00} under the additional assumption that the Hamiltonian is convex in $p$.
By exploiting the metric character of first-order HJ equations, homogenization was extended to quasiconvex Hamiltonians in \cite{DS09, AS13}.
The question of whether homogenization holds in the stationary ergodic setting for coercive Hamiltonians that are nonconvex in the momentum remained open for about fifteen years, until the third author provided in~\cite{Zil} the first counterexample to homogenization in dimensions $d>1$.
Feldman and Souganidis generalized this example and showed in~\cite{FS17} that homogenization can fail for Hamiltonians of the form $H(x,p,\omega)\coloneqq G(p)+V(x,\omega)$ whenever $G$ has a strict saddle point.
This has shut the door to the possibility of having a general qualitative homogenization theory in the stationary ergodic setting in dimensions $d\geqslant 2$ -- at least without imposing further mixing conditions on the stochastic environment -- and stands in sharp contrast to the periodic case, where qualitative homogenization is known to hold for Hamiltonians solely coercive in the momentum, regardless of convexity \cite{LPV}. 

On the positive side, homogenization of \eqref{eq eps hj} for coercive and nonconvex Hamiltonians of fairly general type has been established in one dimension in \cite{ATY_1d, Gao16}, and in any space dimension for Hamiltonians of the form $H(x,p,\omega)=\big(|p|^2-1\big)^2+V(x,\omega)$ in~\cite{ATY_nonconvex}. {This result was generalized in \cite{QTY18}, where the authors studied Hamiltonians of the form $H(x,p,\omega)=\Psi(|p|)+V(x,\omega)$ under suitable monotonicity assumptions on $H$.}
Further positive results in random environments satisfying a finite-range dependence condition were obtained in~\cite{AC18} for Hamiltonians that are positively homogeneous of degree $\alpha\ge 1$. Subsequently, the techniques from that work were adapted in~\cite{FS17} to address Hamiltonians with strictly star-shaped sublevel sets.
Despite this significant progress, the general question of which equations of the form \eqref{eq eps hj} homogenize in the nonconvex case is still not fully understood.
\smallskip

When the coercivity condition of $H$ in $p$ is dropped, one loses control of the derivatives of solutions to \eqref{eq eps hj} and of the associated “cell” problem,  which are no longer Lipschitz continuous  in general.
As a consequence, homogenization of \eqref{eq eps hj} is known to fail even in the periodic case, regardless of whether the Hamiltonian is convex in $p$; see, for instance, the introductions in \cite{Ca10, CNS11} and some  examples in~\cite{BaTe13}.
In this level of generality, additional conditions must be imposed to compensate for the lack of coercivity of the Hamiltonian.
In the periodic and other compact settings, homogenization results of this type have been obtained in~\cite{Ar97, Ar98, AL98, AB10} and, more recently, in certain convex situations in~\cite{Ba24}, for a class of nonconvex Hamiltonians in dimension $d=2$ in~\cite{Ca10}, and in other nonconvex cases in \cite{Ba07, BaTe13}.
When $H(x,p,\omega)\coloneqq |p|+\langle V(x,\omega),p\rangle$, equation \eqref{eq eps hj} is known in the literature (up to a sign change) as the \emph{$G$-equation}. Homogenization has been established both in the periodic setting \cite{CNS11, XiYu10, Sic21} and in the stationary ergodic case \cite{NN11, CaSo13}; see also \cite{Coop22, Coop23} for quantitative results, under a smallness condition on the divergence of $V$, but without imposing $|V|<1$, meaning that $H$ is not assumed to be coercive in $p$.
\smallskip

This paper furnishes a new and fairly general class of nonconvex and noncoercive Hamiltonians for which \eqref{eq eps hj} homogenizes.
{\bluee Our first results, Theorems~\ref{thm:genhom} and \ref{Result: Concentration property}, establish homogenization and a quantitative convergence rate for solutions to \eqref{eq eps hj} with linear initial data} for a class of nonconvex Lipschitz Hamiltonians arising from differential game theory.
Specifically, we consider Hamiltonians of the form
\begin{equation}\label{intro def H}
H(x,p,\omega)\coloneqq \max_{b\in B}\min_{a\in A}\big\{-\ell(x,a,b,\omega)-\langle f(a,b),p\rangle\big\}
\qquad \text{for all } (x,p,\omega)\in \TM\times\Omega, \tag{H}
\end{equation}
where the main assumptions are that the law of $\ell$ has finite-range dependence, in the spirit of \cite{AC18, FS17}, and that there exist a direction $e\in \Sph^{d-1}$ and $\delta>0$ such that
\[
\langle f(a,b),e\rangle \ge \delta \qquad \text{for all } a\in A,\ b\in B. \tag{$f$}
\]
Notably, assumption $(f)$ precludes a Hamiltonian of the form \eqref{intro def H} from being coercive; see Remark~\ref{oss:noncoercive H}. This is a significant point of originality that distinguishes our work from most contributions on stochastic homogenization.
\smallskip

Another important novelty lies in the proof technique.
Indeed, thanks to the form \eqref{intro def H} of the Hamiltonian, we can represent the solution of \eqref{eq eps hj} as the value function of a differential game, as explained in~\cite{EvSo84}, and adopt a game-theoretic approach.
Such an approach has rarely been used in the homogenization of nonconvex HJ equations (see, e.g.,~\cite{BaTe13} in the periodic setting) and, to our knowledge, this is the first time it is employed to obtain a positive result in the stochastic case.
By analyzing optimal strategies, generated paths, and the dynamic programming principle, we show that solutions of \eqref{eq eps hj} exhibit asymptotic concentration and that their mean satisfies an approximate subadditive inequality. {\bluee The homogenization results then follows from the local Lipschitz regularity of solutions to \eqref{eq eps hj}.}
\smallskip

The probabilistic arguments we employ are related to those used in \cite{ACS14, AC18} and to their variation in \cite{FS17}, where the authors prove homogenization for several classes of first- and second-order nonconvex Hamilton–Jacobi equations.
They consider an auxiliary stationary Hamilton–Jacobi equation, the \emph{metric problem}~\cite{ACS14} (respectively, \cite{AC18, FS17}), whose solutions can be interpreted as the minimal cost of going from one point in space to another (respectively, to a planar surface).
By analogy with techniques from first-passage percolation \cite{alexander1993, kesten1993}, they combine Azuma’s inequality with a subadditive argument to prove homogenization of the metric problem and to obtain convergence rates and concentration estimates.
They then use a PDE argument to relate the metric problem to the original Hamilton–Jacobi equation.
In comparison, our proof presents several key differences. First, 
the concentration and subadditive techniques are applied to the value of a two-player zero-sum differential game, rather than to the cost of an optimal control considered in the metric problem. 
Indeed, the metric problem can be seen as a degenerate two-player game ({\bluee where} Player 2 has no
actions), i.e., an optimal-control formulation. While this usually produces convex Hamiltonians,
Armstrong–Cardaliaguet [7] showed that, under positive homogeneity,
the metric problem extends to certain nonconvex cases, with homogenization
obtained via quantitative concentration rather than exact subadditivity.
Secondly, our arguments rely primarily on a game-theoretic approach, exploiting the monotonicity (in the preferred direction $e$) of optimal trajectories, rather than on PDE methods.
Thirdly, we treat noncoercive Hamiltonians, whereas \cite{ACS14, AC18, FS17} assume coercivity.
This leads to several difficulties, including the fact that the spatial Lipschitz constants of solutions to \eqref{eq eps hj} are not uniformly bounded with respect to $\varepsilon$.
\smallskip

{\bluee As an interesting output of the quantitative homogenization rate in Theorem~\ref{Result: Concentration property}, we show that the homogenization results described above extend to Hamiltonians that arise as local uniform limits of suitable sequences of Hamiltonians of the form (H), see Corollary~\ref{cor:quantitative estimate}, and that, in general, need not admit the same max–min representation. A further result in this direction is given by Theorem~\ref{thm:genhom2}, where we extend homogenization to a class of nonconvex and noncoercive Lipschitz Hamiltonians that are not necessarily given by a max–min formula. 
This makes the game-theoretic approach even more notable, as it applies to Hamiltonians that do not \emph{a priori} arise from a differential game.}
For this extension, we adapt the argument introduced in \cite[Section~5]{EvSo84} to put these Hamiltonians into the form (H) when $p$ is constrained within a ball $B_R$, but using it to prove homogenization in the noncoercive setting is nontrivial and, as far as we know, new.
The difficulty lies in the fact that, due to the lack of coercivity of the Hamiltonian, the Lipschitz constants in $x$ of solutions to \eqref{eq eps hj} are not uniformly bounded in $\eps>0$, but instead blow up at rate $1/\eps$. 
In view of this, we tailored the proof of Theorem~\ref{thm:genhom} to this extension, ensuring that the constants appearing in the crucial estimates underpinning our arguments depend only on parameters that remain controlled when we perform the localization argument.

Our work is closely connected to the joint paper~\cite{GZ23} of the third author. There, the authors introduced a new model of discrete-time games, called \textit{percolation games}. They established a condition, called “oriented assumption”, under which the value of the $n$-stage game converges as $n \to \infty$. Moreover, they sketched a heuristic link between the existence of such a limit and stochastic homogenization, explaining how assumptions on the discrete game can be translated into assumptions on Hamiltonians. 
The present paper provides the first formal implementation of this program: we identify precisely which Hamiltonians correspond to “oriented games” and turn the convergence result for oriented games into a rigorous result in stochastic homogenization. In this sense, our paper constitutes the first “proof of concept” that the methodology outlined in~\cite{GZ23} can be fully validated. We refer the reader to~\cite[Section~4]{GZ23} for a detailed presentation of the methodology.
While the proof of Proposition~\ref{prop:expectation}, which constitutes the central result of our paper, shares several ingredients with Theorem~2.3 in~\cite{GZ23}, notably the use of concentration inequalities and subadditivity, the differential game and Hamilton–Jacobi framework calls for substantially different techniques. In particular, viscosity solutions and comparison principles play a central role, and their use is especially delicate here due to the non-coercive nature of the Hamiltonians under consideration.

\smallskip

{\bluee
The paper is organized as follows. 
In Section~\ref{sec:assumptions} we present the notation, the standing assumptions and the statements of our homogenization results, namely Theorems~\ref{thm:genhom}, \ref{Result: Concentration property} and \ref{thm:genhom2} and Corollary~\ref{cor:quantitative estimate}. 
In Section~\ref{sec:outline} we present the reduction  strategy we will follow to prove these results. 
Some proofs are deferred to Appendix~\ref{Section: Reduction}. 
In Section~\ref{sec:concentration} we prove the probabilistic concentration result. 
Section~\ref{sec:proof} is devoted to the proofs of Theorems~\ref{thm:genhom}, \ref{Result: Concentration property} and \ref{thm:genhom2} and Corollary~\ref{cor:quantitative estimate}. Appendix~\ref{sec: app A} contains the deterministic PDE results, along with their proofs, that we use in the paper. 
}
\medskip 

{\textsc{Acknowledgments.}}
{\bluee 
We thank the anonymous referees for their careful reading and many valuable comments. In particular, we are grateful to one referee for recognizing that our proof strategy could be leveraged to obtain, essentially for free, a quantitative rate of homogenization, and for suggesting the stability result leading to Corollary~\ref{cor:quantitative estimate}.
AD is a member of the INdAM Research Group GNAMPA.
He thanks Guy Barles for a long and valuable email exchange in the fall of 2023 concerning possible generalizations of the comparison principle stated in Theorem~\ref{appA: teo comp}.
AD is particularly grateful to Guy Barles for carefully reading his attempts to generalize Theorem~\ref{appA: teo comp} and for sharing his expertise on the matter, including the details needed to turn the argument sketched in~\cite[Remark~5.3]{barles2} into a full proof.
}
BZ is very grateful to Scott Armstrong and Pierre Cardaliaguet for all the enlightening discussions on homogenization theory. 
This work was supported by Sapienza Universit\`a di Roma - Research Funds 2018 and 2019, by the French Agence Nationale de la Recherche (ANR) under reference ANR-21-CE40-0020 (CONVERGENCE project), and by the ERC CoG 863818 (ForM-SMArt) grant. 
It was partly done during a 1-year visit of BZ to the Center for Mathematical Modeling (CMM) at University of Chile in 2023, under the IRL program of CNRS.\medskip

\numberwithin{teorema}{section}
\numberwithin{equation}{section}

\section{Assumptions and main results}\label{sec:assumptions}

Throughout the paper, we will denote by $d \in \N$ the dimension of the ambient space. 
We will denote either by $B_r(x_0)$ or $B(x_0,r)$ (respectively, $\overline B_r(x_0)$ or $\overline B(x_0,r)$)  the open (resp., closed) ball in $\R^d$ of radius $r > 0$ centered at $x_0 \in \M$. 
When $x_0 = 0$, we will more simply write $B_r$ (resp., $\overline B_r$).
The symbol $| \cdot |$ will denote the norm in $\R^k$, for any $k \geqslant 1$.
We will write $\varphi_n \ucv \varphi$ in $E \subseteq \R^k$ to mean that the sequence of functions $(\varphi_n)_n$ uniformly converges to $\varphi$ on compact subsets of $E$. 
We will denote by $\CC(X), \,\UC(X), \,\BUC(X)$, and $\Lip(X)$ the space of continuous, uniformly continuous, bounded uniformly continuous, and Lipschitz continuous functions on a metric space $X$, respectively. 
\smallskip

We will denote by $(\Omega, \F, \PP)$ a probability space, where $\P$ is a probability measure and ${\mathcal F}$ is the $\sigma$-algebra of $\P$--measurable subsets of $\Omega$. 
We will assume that $\P$ is {\em complete} in the usual measure theoretic sense. 
We will denote by ${\mathcal B}(\R^k)$ the Borel $\sigma$-algebra on $\R^k$ and equip the product spaces $\M \times \Omega$ and $\M \times A \times B \times \Omega$ with the product $\sigma$-algebras ${\mathcal B(\M)} \otimes {\mathcal F}$ and ${\mathcal B(\M)}\otimes \mathcal B(\R^m)\otimes \mathcal B(\R^m)\otimes {\mathcal F}$, respectively.
\smallskip

We will assume that $\P$ is invariant under the action of a one-parameter group $(\tau_x)_{x \in \M}$ of transformations $\tau_x \colon \Omega \to \Omega$. 
More precisely, we assume that: 
the mapping $(x, \omega) \mapsto \tau_x\omega$ from $\M \times \Omega$ to $\Omega$ is measurable; 
$\tau_0 = id$; 
$\tau_{x + y} = \tau_x \circ \tau_y$ for every $x, y \in \M$; 
and $\P \left( \tau_x (E) \right) = \P( E )$, for every $E \in{\mathcal F}$ and $x \in \M$. 
Lastly, we will assume that the action of $(\tau_x)_{x \in \M}$ is {\em ergodic}, i.e., any measurable function $\varphi \colon \Omega \to \R$ satisfying $\P( \varphi( \tau_x \omega ) = \varphi( \omega ) ) = 1$ for every fixed $x \in {\M}$ is almost surely equal to a constant.
\smallskip

A random process $f \colon \M \times \Omega \to \R$ is said to be \emph{stationary} with respect to $(\tau_x)_{x \in \M}$ if $f(x, \omega) = f(0, \tau_x \omega)$ for all $(x, \omega) \in \M \times \Omega$. 
Moreover, whenever the action of $(\tau_x)_{x \in \M}$ is ergodic, we refer to $f$ as a \emph{stationary ergodic} process.
\smallskip

Let $(X_i)_{i \in \I}$ be a (possibly uncountable) family of jointly measurable functions from $\M \times \Omega$ to $\R$. 
We will say that the random variables $(X_i)_{i \in \I}$ exhibit {\em long-range independence} (or, equivalently, have {\em finite range of dependence}) if there exists $\rho>0$ such that, for all pair of sets $S, \widehat{S} \subseteq \M$ such that their Hausdorff distance $d_H(S, \widehat{S}) > \rho$, the generated $\sigma$--algebras $\sigma(\{ X_i(x,\cdot) : i \in \I,  x \in S \})$ and $\sigma(\{ X_i(x,\cdot) : i \in \I, x \in \widehat S\})$ are independent, in symbols,         
\begin{equation}\label{def R range dependence}
	\sigma(\{ X_i(x,\cdot) : i \in \I,  x \in S \})
	\indep 
	\sigma(\{ X_i(x,\cdot) : i \in \I,  x \in \widehat S \})
	\qquad
	\hbox{whenever\ \ $d_H(S, \widehat{S}) \ge \rho$}.
	\tag{FRD}
\end{equation}

In this paper, we will be concerned with the Hamilton-Jacobi equation of the form
\begin{equation}\label{eq:generalHJ}
	\partial_t u+H(x,D_x u,\omega)=0,\quad\hbox{in}\ \Tcyl,
\end{equation}
where the Hamiltonian $H \colon \R^d \times\R^d \times \Omega \to \R$ is assumed to be stationary with respect to shifts in $x$ variable, i.e., $H(x + y, p, \omega) = H(x, p, \tau_y \omega)$ for every $x, y \in \R^d$, $p \in \R^d$, and $\omega \in \Omega$, and to belong to the class $\Ham$ defined as follows. 
\begin{definition}\label{def:Ham}
	A function $H \colon \R^d\times\R^d\times\Omega\to\R$ is said to be in the class $\Ham$ if it is jointly measurable and it satisfies the following conditions, for some constant  $\beta>0$:
	\begin{itemize}
		\item[(H1)] 
			$|H(x, p, \omega)| \leqslant \beta \left( 1 + |p| \right) 
			\qquad \qquad \qquad \qquad\  \text{for all } (x, p) \in \TM$;
			\smallskip
		\item[(H2)] 
			$|H(x, p, \omega) - H(x, q, \omega)| \leqslant \beta |p - q| 
			\qquad\qquad \text{for all } x, p, q\in\M$;
			\smallskip
		\item[(H3)] 
			$|H(x, p, \omega) - H(y, p, \omega)| \leqslant \beta |x - y| 
			\qquad \quad \quad \text{for all } x, y, p\in\M$.
	\end{itemize}
\end{definition}

Assumptions (H1)-(H3) guarantee well-posedness in $\CC(\cTcyl)$, for every fixed $T>0$, of the Cauchy problem associated with equation \eqref{eq:generalHJ} when the initial datum is in $\UC(\M)$. 
Furthermore, the solutions are actually in $\UC(\cTcyl)$. 
Solutions, subsolutions and supersolutions of \eqref{eq:generalHJ} will be always understood in the viscosity sense, see \cite{bardi, barles_book, barles2, users}, and implicitly assumed continuous, if not otherwise specified.
\smallskip

The purpose of this paper is to prove a homogenization result for equation~\eqref{eq:generalHJ} for a subclass of stationary Hamiltonians belonging to $\Ham$ that arise from Differential Game Theory and that can be expressed in the following max-min form:
\begin{equation}\label{def H}
\tag{H}
	H(x,p,\omega) 
		\coloneqq \max_{b \in B} \min_{a \in A} \left\{ -\ell(x, a, b,\omega) - \langle f(a, b), p \rangle \right\} 
		\qquad \text{for all } (x, p, \omega) \in \TM\times\Omega.
\end{equation}
Here $A, B$ are compact subsets of $\R^m$, for some integer $m$, and the product space $\M \times A \times B \times \Omega$ is equipped with the product $\sigma$-algebra ${\mathcal B(\M)} \otimes \mathcal B(\R^m) \otimes \mathcal B(\R^m) \otimes {\mathcal F}$. 
{\bluee 
The mapping $f \colon A \times B \to \R^d$ is a continuous {\bluee vector--valued function} satisfying the following assumption:
\begin{itemize}	
\item[($f$)] (oriented dynamics)  
		the dynamics given by $f \colon A \times B \to \M$ is \emph{oriented}, i.e., there exists $\delta > 0$ and a direction $e \in \Sph^{d-1}$ such that 
		\[
			\langle f(a, b), e \rangle \geqslant \delta \qquad \hbox{for all $(a, b) \in A \times B$.}
		\] 
\end{itemize}
For the {\em running cost} $\ell \colon \R^d \times A \times B \times \Omega \to \R$, we will assume it is jointly measurable and satisfies:
\begin{itemize}
    \item[($\ell_1$)] \ 
    	$\ell(\cdot, \cdot, \cdot, \omega) \in \BUC(\R^d \times A \times B)$ \quad for every $\omega \in \Omega$;
    	\smallskip 
    \item[($\ell_2$)] \ 
    	there exists a constant $\Lip(\ell) > 0$ such that 
		\[
			|\ell(x, a, b, \omega) - \ell(y, a, b, \omega)| \leqslant \Lip(\ell)\, |x - y|
			\qquad \hbox{for all $x, y \in \R^d$, $a \in A$, $b \in B$ and $\omega \in \Omega$;}
		\]
		\item[($\ell_3$)] \ 
			$\ell$ is stationary with respect to $x$, i.e., 
		\[
			\ell(x, a, b, \omega) = \ell(0, a, b, \tau_x \omega)
			\qquad
			\hbox{for all $x \in \M$, $a \in A$, $b \in B$ and $\omega \in \Omega$.}
		\]
\end{itemize}
}
Throughout the paper, we will denote by $\Hamdg$ the subclass of Hamiltonians in $\Ham$ that can be put in the form \eqref{def H} with $f$ and $\ell$ satisfying assumptions ($f$) and ($\ell_1$)--($\ell_2$), respectively. 
A Hamiltonian $H$ belonging to $\Hamdg$ will be furthermore termed stationary to mean that assumption ($\ell_3$) is in force.  In the sequel, we shall denote by $\| \ell \|_\infty$ the $L^\infty$--norm of $\ell$ on $\R^d \times A \times B \times \Omega$, which is finite due to ($\ell_1$), ($\ell_3$), and the ergodicity assumption on $\Omega$.
\smallskip 

{\bluee
The proof of our homogenization result relies crucially on the oriented-dynamics assumption $(f)$  and on the following long--range independence hypothesis:
\begin{itemize}
	\item[($\ell_4$)] (long-range independence) 
		the random variables $\left( \ell(\cdot, a, b, \cdot) \right)_{(a, b) \in A \times B}$ from $\M \times \Omega$ to $\R$ exhibit long-range independence, i.e., there exists $\rho > 0$ such that \eqref{def R range dependence} holds with $\I \coloneqq A \times B$ and $X_i \coloneqq \ell(\cdot, a, b, \cdot)$ where $i = (a, b)$.
		\medskip
\end{itemize}
\smallskip
	    
The specific form \eqref{def H} of the Hamiltonian allows to represent solutions to equation \eqref{eq:generalHJ} via suitable formulae issued from Differential Games, see~\cite{EvSo84}. Indeed, let us denote by 
\begin{align*}
	\AA(T) \coloneqq \left\{ a \colon [0,T] \to A \, : \, a \text{ measurable} \right\},\qquad
	\BB(T) \coloneqq \left\{ b \colon [0,T] \to B \, : \, b \text{ measurable} \right\}.
\end{align*}
The sets $A$ and $B$ are to be regarded as action sets for Player~1 and~2, respectively. 
A {\em nonanticipating strategy} for Player~1 is a function $\alpha \colon \BB(T) \to \AA(T)$ such that, 
for all $b_1, b_2 \in \BB(T)$ and $\tau \in [0, T]$, 
\[
	b_1(\cdot) = b_2(\cdot) \quad \text{in } [0,\tau]
		\quad
		\Rightarrow
		\quad
	\alpha[b_1] (\cdot) = \alpha[b_2](\cdot) \quad \text{in } [0,\tau] \,.
\]
We will denote by $\Gamma(T)$ the family of such nonanticipating strategies for Player~1. 
For every fixed $\omega\in\Omega$ and every $(t,x)\in\cyl$, let us set  
\begin{equation}
\label{def value function}
	v(t,x,\omega) \coloneqq \sup_{\alpha \in \Gamma(t)}\inf_{b \in \BB(t)} \left\{
		\int_0^t \ell(y_x(s),\alpha[b](s),b(s),\omega)\,ds + g(y_x(t)) 
	\right\},
\end{equation}
where $y_x \colon [0,t] \to \M$ is the solution of the ODE 
\begin{align}\tag{ODE}\label{eq ODE}
	\begin{cases}
		\dot y_x(s) = f(\alpha[b](s),b(s))\qquad\hbox{in $[0,t]$} \\
		y_x(0) = x.
	\end{cases}
\end{align} 
The function $v$ defined by \eqref{def value function} is usually called {\em value function}. It is the unique continuous viscosity solution of the unscaled HJ equation \eqref{eq eps hj} (i.e., with $\eps=1$) satisfying the initial condition $v(0,\cdot,\omega)=g$ on $\M$ for every $\omega\in\Omega$. We refer the reader to Appendix~\ref{app:PDE} for more details and relevant results.
\smallskip
}

Our main result reads as follows.

\begin{theorem}\label{thm:genhom}
	Let $H$ be a stationary Hamiltonian belonging to $\Hamdg$ and satisfying hypotheses ($\ell_4$) and ($f$). 
	Then, the HJ equation \eqref{eq eps hj} homogenizes, i.e., there exists a continuous function $\HV \colon \M \to \R$, called {\em the effective Hamiltonian}, and a set $\hat\Omega$ of probability 1 such that, for every fixed $\omega\in\hat\Omega$ and every $g \in \D{UC}(\R^d)$, the solutions $u^\epsilon(\cdot, \cdot, \omega)$ of \eqref{eq eps hj} satisfying $u^\epsilon(0, \,\cdot\,, \omega) = g$ converge, locally uniformly on $\cTcyl$ as $\epsilon \to 0^+$, to the unique solution $\ol{u}$ of 
    \begin{eqnarray*}
	   \begin{cases}
	    \partial_t \ol{u} +\HV( D_x\ol{u}) = 0 
	    	& \hbox{in $\Tcyl$} \\
	    \ol{u}(0,\,\cdot\,) = g 
	    	& \hbox{in $\R$}.
		\end{cases}
    \end{eqnarray*}
	Furthermore, $\HV$ satisfies (H1) and (H2).
\end{theorem}

\begin{remark}\label{oss:noncoercive H}
	We stress that a Hamiltonian of the form \eqref{def H} with $f$ satisfying condition ($f$) is never coercive. Indeed,
	\[
		\lim_{\lambda \to -\infty} H(x, \lambda e, \omega) = +\infty,
		\quad 
		\lim_{\lambda \to +\infty} H(x, \lambda  e, \omega) = -\infty
		\qquad\hbox{for every $(x, \omega) \in \M \times \Omega$.} 
	\]
\end{remark}
{\bluee
Theorem~\ref{thm:genhom} is actually derived from the following stronger quantitative result.

\begin{theorem} 
\label{Result: Concentration property}
Let $H$ be a stationary Hamiltonian belonging to $\Hamdg$ and satisfying hypotheses ($\ell_4$) and ($f$). 
Let us denote by $\tilde u^\eps_\theta$ the solution  of \eqref{eq eps hj} satisfying $\tilde u^\epsilon(0,x, \omega) = \langle \theta, x\rangle$ for all $(x,\omega)\in\M\times\Omega$. 
Then there exists a deterministic function $\overline H \colon \M \to \R$ such that, for every fixed $\theta \in \M$, $T>0$ and $R>0$,  we have 
\begin{equation}\label{claim concentration inequality}
   \m{P}\left( \sup_{[0,T]\times B_{R}} \left |{\tilde u^\eps_\theta(t, x, \omega)}-\langle \theta, x\rangle + t\overline{H}(\theta)\right| \geq 
   K  \left(-\eps {\ln \eps}\right)^{1/2} \right)  \leq \eps^{2}    
\quad
\hbox{for all $\eps\leq 1/2$,}
\end{equation}
for some constant $K$ depending on $R$,\,$T$\,$|\theta|$,\,$d$,\,$\beta$,\,$\rho,\delta,\, \Lip(\ell)$ and $\left\|f\right\|_\infty$.
\end{theorem}

This quantitative estimate yields the following interesting consequence. 

\begin{cor}\label{cor:quantitative estimate}
Let $(H_n)_n$ be a sequence of stationary Hamiltonians belonging to $\Hamdg$ and satisfying hypotheses ($\ell_4$) and ($f$). Let us assume that the associated quantities $\beta_n$,\,$\rho_n$,\,$\delta_n$,\, $\Lip(\ell_n)$ and $\left\|f_n\right\|_\infty$ satisfy the following bounds:
\begin{equation}\label{eq:equi-bounds}
C \coloneqq \sup_n\left(\beta_n+\rho_n+\Lip(\ell_n)+\left\|f_n\right\|_\infty\right)<+\infty,
\qquad
\delta \coloneqq \inf_n \delta_n>0.
\end{equation}
If $H_n(\cdot,\cdot,\omega)\ucv H(\cdot,\cdot,\omega)$ in $\TM$ for every $\omega\in\Omega$, the following holds:
\begin{enumerate}[(i)]
\item there exists $\overline H:\M\to\R$ satisfying (H1), (H2) with $\beta \coloneqq \sup_n\beta_n$ such $\overline H_n\ucv \overline H$ in $\R^d$;\smallskip
\item  for every $\theta\in\M$, the solution $\tilde u^\eps_\theta$ of \eqref{eq eps hj} subject to $\tilde u^\epsilon(0,x, \omega) = \langle \theta, x\rangle$ for all $(x,\omega)\in\M\times\Omega$ satisfies \eqref{claim concentration inequality}, for every fixed $T>0$ and $R>0$, for some constant $K$ depending on $R$,\,$T$\,$|\theta|$,\,$d$,\,$C$,\,$\delta,\, \Lip(\ell)$.
\end{enumerate}
In particular, the HJ equation \eqref{eq eps hj} homogenizes with effective Hamiltonian $\overline H$.
\end{cor}

We emphasize that the limiting $H$ above is a stationary Hamiltonian belonging to $\Ham$, but it  need not lie in $\Hamdg$; namely, it cannot, in general, be written in the max-min form \eqref{def H}.\smallskip
}

{\bluee 
By combining suitable Lipschitz bounds for solutions to \eqref{eq eps hj} with Lipschitz initial data with a localization argument inspired by \cite[Section 5]{EvSo84}, we further establish the homogenization results above for a different subclass of Hamiltonians in $\Ham$ that intersects, but is not contained in, $\Hamdg$. This subclass is described in the next theorem.
}

\begin{theorem}\label{thm:genhom2}
	Let $G$ be a stationary Hamiltonian belonging to $\Ham$ and satisfying the following assumption:
	\begin{itemize}
		\item[(G1)] the random variables $\left( G(\cdot, p, \cdot) \right)_{p \in \M}$ from $\M \times \Omega$ to $\R$ exhibit long-range independence, i.e., there exists $\rho > 0$ such that \eqref{def R range dependence} holds with $\I \coloneqq \M$ and $X_i \coloneqq G(\cdot, p, \cdot)$ where $i = p$. 
	\end{itemize}
{\bluee	Then, the quantitative estimate stated in \Cref{Result: Concentration property} and the homogenization result stated in Theorem~\ref{thm:genhom} hold} for any Hamiltonian $H$ of the form 
	$H(x, p, \omega) \coloneqq G(x, \pi(p), \omega) + \langle p, v \rangle$, where $v$ is a nonzero vector in $\M$ and $\pi \colon \M \to \M$ is a linear map such that $\pi(v) = 0$. 
\end{theorem}

Examples of Hamiltonians $G$ lying in $\Ham$ and satisfying $(G1)$ are the ones of the form $G(x, p, \omega) \coloneqq G_0(p) + V(x, \omega)$, where $G_0$ belongs to $\Ham$ and $V \colon \M \times \Omega \to \R$ is a stationary function, globally bounded and Lipschitz on $\M$, which satisfies \eqref{def R range dependence} with $\I \coloneqq \left\{ 0 \right\}$ and $X_0 \coloneqq V$.

\section{Reduction arguments for homogenization}\label{sec:outline}

In this section we describe the reduction strategy that we will follow to prove Theorem~\ref{thm:genhom}.
The first step consists in noticing that, in order to prove homogenization for equation \eqref{eq eps hj}, it is enough to restrict to linear initial data instead of any $g\in\UC(\M)$. 
The precise statement is the following. 

\begin{theorem}\label{teo reduction}
	Let $H$ satisfy hypotheses (H1)-(H3) and denote by $\tilde u^\eps_\theta$ the unique continuous solution of equation \eqref{eq eps hj} satisfying $\tilde u^\eps_\theta(0, x, \omega) = \langle \theta, x \rangle$ for all $(x, \omega) \in \M \times \Omega$ and for every fixed $\theta \in \M$ and $\eps > 0$.
	Assume there exists a function  $\overline H \colon \M \to \R$ such that, for every $\theta \in \M$, the following convergence takes place for every $\omega$ in a set $\Omega_\theta$ of probability $1$:
	\begin{equation}\label{app hom linear data}
		\tilde u^\eps_\theta(t, x, \omega) \ {\ucv} \ \langle \theta, x \rangle - t \overline H(\theta)
		\quad\hbox{in $\cTcyl$ as $\epsilon \to 0^+$.}
	\end{equation}
	Then, $\overline H$ satisfies condition (H1)-(H2). 
	Furthermore, there exists a set $\hat\Omega$ of probability 1 such that, for every fixed $\omega\in\hat\Omega$ and every $g \in \D{UC}(\R^d)$,  the unique function $u^\eps(\cdot, \cdot, \omega) \in \CC(\cTcyl)$ which solves \eqref{eq eps hj} with initial
	condition $u^\eps(0, \cdot, \omega) = g$ in $\M$ converges, locally uniformly in $\cTcyl$ as $\eps \to 0^+$, to the unique solution $\overline u \in \D{C}(\cTcyl)$ of
	\begin{equation}\label{app effective eq}
		\partial_t \overline u + \overline H(D_x \overline u) = 0
		\qquad\hbox{in $\Tcyl$}
	\end{equation}
	with the initial condition $\overline u(0, \cdot) = g$.
\end{theorem}

This reduction argument, which is essentially of deterministic nature, was already contained in the pioneering work~\cite{LPV} on periodic homogenization, at least as far as first-order HJ equations are concerned. 
This holds, in fact, even in the case when equation \eqref{eq eps hj} presents an additional (vanishing) diffusive and possibly degenerate term. 
A proof of this can be found in~\cite{DK17} and is given for Hamiltonians that are coercive in the $p$-variable. 
Such a class does not {\bluee cover} the kind of Hamiltonians we consider here, as pointed out in Remark~\ref{oss:noncoercive H}. 
Yet, the extension follows by arguing as in~\cite{DK17}, with the only difference that one has to use a different Comparison Principle, namely Theorem~\ref{appA: teo comp2}, in place of \cite[Proposition 2.4]{DK17}.  
We refer the reader to \Cref{Section: Reduction} for the detailed argument.
\smallskip

Theorem~\ref{teo reduction} yields, in particular, that the effective Hamiltonian $\overline H$ is identified by the following almost sure limit:
\[
	\overline H(\theta) \coloneqq - \lim_{\epsilon \to 0} \tilde u^\epsilon_\theta(1, 0, \omega)
	\qquad\hbox{for every fixed $\theta \in \M$}.
\]

The second step in the reduction consists in observing that, in order to prove the local uniform convergence required to apply \Cref{teo reduction}, it is enough to prove it for a fixed value of the time variable, that we chose equal to $1$. 

\begin{lemma}
\label{lemma reduction}
	Let $\omega \in \Omega$ and $\theta \in \R^d$ be fixed, and assume that
	\begin{equation}\label{t1}
		\limsup_{\epsilon\to 0^+} \sup_{y\in B_R} |\tilde u^\epsilon_\theta(1, y,\omega) - \langle \theta, y \rangle + \overline{H}(\theta)| = 0
		\qquad \hbox{for every $R>0$.}
	\end{equation}	
	Then, for every $T > 0$,  
	\begin{equation}\label{claim distilled}
		\tilde u_\theta^\epsilon(t,x,\omega) \ucv \langle \theta, x \rangle - t \overline{H}(\theta)
		\quad\text{in }\cTcyl.
	\end{equation}
\end{lemma}
 
\begin{proof}
	Since $\omega$ will remain fixed throughout the proof, we will omit it from our notation. 
	Let us fix $\theta \in \R^d$.  	
	We first take note of the following scaling relations
	\[
		\tilde u^\epsilon_\theta(t, x) 
			= \eps \tilde u^1_\theta(t / \epsilon, x / \epsilon)
			= t (\epsilon / t) \tilde u^1_\theta(t / \epsilon, x / \epsilon)
			= t  \tilde u^{\epsilon/t}_\theta(1, x / t) 
			\qquad\text{for all $t > 0$ and $x \in \R^d$}.
	\]
	Fix $T > 0$. 
	Then, for every fixed $r \in (0, T)$, we obtain
	\begin{align}\label{inequality distilled}
		\sup_{r \leqslant t \leqslant T} \sup_{y \in B_R} \left | \tilde u^\epsilon_\theta(t, y) - \langle \theta, y \rangle + t \overline{H}(\theta) \right|
			&= \sup_{r \leqslant t \leqslant T} \; \sup_{y \in B_R} 
			\left | t \big(\tilde u^{\epsilon / t}_\theta(1, y / t) - \langle \theta, y / t \rangle + \overline{H}(\theta)\big) \right | \nonumber \\
			&\leqslant T \sup_{\epsilon / T \leqslant \eta \leqslant \epsilon / r} \; \sup_{z \in B_{R / r}} \left | \tilde u^\eta_\theta(1, z) - \langle \theta, z \rangle + \overline{H}(\theta) \right | \,.
	\end{align}
	By \eqref{t1}, the right-hand side goes to $0$ as $\epsilon\to 0^+$. 
	On the other hand, in view of Proposition~\ref{prop properties value function}-(i) and of the fact that $\tilde u^\epsilon_\theta(0, x) = \langle \theta, x \rangle$ for all $x \in \M$, we have
	\begin{align*}
		\sup_{0 \leqslant t \leqslant r} \; \sup_{y \in \R^d} \left | \tilde u^\epsilon_\theta(t, y) - \langle \theta, y \rangle + t \overline{H}(\theta) \right |
			&\leqslant r |\overline H(\theta)| + \sup_{0 \leqslant t \leqslant r} \; \sup_{y\in \R^d} \left | \tilde u^\epsilon_\theta(t, y) - \langle \theta, y \rangle \right | \\
			&\leqslant r \left(|\overline H(\theta)| + \|\ell\|_\infty + \|f\|_\infty\right) \,.
	\end{align*}
	Assertion \eqref{claim distilled} follows from this and \eqref{inequality distilled} by the arbitrariness of the choice of $r \in (0, T)$. 
\end{proof}

In order to simplify some arguments, we find it convenient to work with solutions with zero initial datum. 
We can always reduce to this case, without any loss of generality, by setting $u^\eps_\theta(t, x, \omega) \coloneqq \tilde u^\eps_\theta(t, x, \omega) - \langle \theta, x \rangle$ for all $(t, x, \omega) \in \cTcyl \times \Omega$. 
The function $u^\eps_\theta$ is the unique continuous function which solves equation \eqref{eq eps hj} with $H_\theta \coloneqq H(\cdot, \theta + \cdot, \omega)$ in place of $H$ and which satisfies the initial condition $u^\eps_\theta(0, x, \omega) = 0$ for all $(x, \omega) \in \M \times \Omega$. 
Note that the Hamiltonian $H_\theta$ is still given by the max-min formula \eqref{def H} where $\ell$ is replaced by $\ell_\theta(x, a, b, \omega) \coloneqq \ell(x, a, b, \omega) + \langle f(a, b), \theta \rangle$. 
Furthermore, $\ell_\theta$ satisfies the same conditions ($\ell_1$)--($\ell_4$).
\smallskip 

The last reduction remark consists in noticing that the following rescaling relation holds
\[
	u^\epsilon_\theta(1, x, \omega) 
		= \epsilon  u_\theta(1 / \epsilon, x / \epsilon, \omega)
		\qquad\text{for all $x \in \R^d$ and $\eps > 0$,}
\]	
where we have denoted by $u_\theta$ the function $u^\eps_\theta$ with $\eps = 1$.
\smallskip

In the light of all this, the proof of \Cref{thm:genhom} is thus reduced to show that,  for every fixed $\theta \in \M$, there exists a set $\Omega_\theta$ of probability $1$ such that, for every $\omega \in \Omega_\theta$, we have 
\begin{equation}\label{claim homogenization}
	\limsup_{\epsilon \to 0^+} \; \sup_{y \in B_R} | u^\epsilon_\theta(1, y, \omega) + \overline{H}(\theta) | 
	= 
	\limsup_{t \to +\infty} \; \sup_{y \in B_{t R}} \left | \frac{u_\theta(t, y, \omega)}{t} + \overline{H}(\theta)\right |
	=
	0
	\quad
	\hbox{for all $R > 0$,}
\end{equation}	
for some deterministic function $\overline H \colon \M \to \R$.
\smallskip

\section{Probabilistic concentration} \label{sec:concentration}

In this section we shall prove, by making use of Azuma's  martingale inequality, that $u_{\theta}(t, 0, \cdot)$ is concentrated. 
For this, we will take advantage of the fact that $u_\theta$ can be expressed via the Differential Game Theoretic formula \eqref{def value function} with initial datum $g \equiv 0$ and running cost $\ell_\theta(x, a, b, \omega) \coloneqq \ell(x, a, b, \omega) + \langle f(a, b), \theta \rangle$. 
Here is where assumptions \eqref{def R range dependence} and $(f)$ play a crucial role, by ensuring altogether that the solution of \eqref{eq eps hj} is robust with respect to local perturbations of the cost function $\ell$. 
\smallskip

{Throughout this section, we will weaken the conditions on the running cost $\ell \colon \R^d \times A \times B \times \Omega \to \R$ and assume that $\ell$ is only jointly measurable and may therefore fail to satisfy conditions ($\ell_1$)-($\ell_2$).}
In particular, no continuity and stationarity conditions with respect to $x$ will be required. 
\smallskip

We start by recalling a classic theorem on concentration of martingales, also known as Azuma's  inequality.
\begin{lemma}[Concentration of martingales \cite{azuma1967WeightedSumsCertain, hoeffding1963ProbabilityInequalitiesSums}]
\label{Result: Concentration of martingales}
    Let $(X_n)_{n \in \NN}$ be a martingale and $(c_n)_{n \in \NN}$ a real sequence such that, for all $n \in \NN$, $|X_n - X_{n + 1}| \leqslant c_n$ almost surely. 
    Then, for all $n \in \NN$ and $M > 0$,
    \[
        \PP( |X_n - X_0| \geqslant M ) \leqslant 2 \exp \left( \frac{-M^2}{2 \sum_{m = 0} ^{n-1} c_m^2} \right) \,.
    \]
\end{lemma}

The probabilistic concentration result mentioned before is stated as follows.
\begin{proposition} 
\label{prop:concentration}
	There exists a constant $c=c(\rho,\delta, \Lip(\ell),\left\|f\right\|_\infty)>0$, only depending on $\rho$, $\delta$, $\Lip(\ell)$ and $\left\|f\right\|_\infty$, such that, for all $M > 0$ and $t \geqslant 1$,
	\begin{equation*}
		\m{P} \left( |u_{\theta}(t, 0,\cdot) - U_\theta(t)| \geqslant M \sqrt{t} \right) \leqslant \exp \left( -c M^2 \right) \,,
	\end{equation*}
	where $U_\theta(t) \coloneqq \EE[u_{\theta}(t, 0, \cdot)]$ denotes the expectation of the random variable  $u_{\theta}(t, 0, \cdot)$.
\end{proposition}

\begin{remark}
We have tailored the proof of Proposition \ref{prop:concentration} in such a way that the constant $c$ appearing in the statement does not depend on  $\|\ell\|_\infty$. This is crucial in view of the homogenization result provided in Theorem \ref{thm:genhom2}. 
\end{remark}

To prove \Cref{prop:concentration}, we need a technical lemma first. 
The result is deterministic, hence the dependence on $\omega$ will be omitted. 
It expresses that, thanks to the condition ($f$) on the dynamics, we can control the variation of the value function as we change the running cost on a strip that is orthogonal to the direction $e$. 

\begin{lemma} \label{lem:strip}
\label{Result: Deterministics payoff dependence on strips of the $T$-horizon value}
	For any given pair $R < \widehat{R}$ in $\R$, let us define the strip between $R$ and $\widehat{R}$ as follows:
    \[
        S_{R, \widehat{R}} \coloneqq \{ x : \langle x, e\rangle \in [R, \widehat{R}] \} \,.
    \]
	Let $\ell, \widehat{\ell} \colon \RR^d \times A \times B \to \RR$ be Borel--measurable bounded running costs and let $f \colon A \times B \to \M$ be a continuous { vector--valued function} satisfying condition ($f$). 
	For every fixed $\theta \in \R^d$, let us denote by $u_\theta(t, x)$, $\widehat u_\theta(t, x)$ the value functions defined via \eqref{def value function} with initial datum $g \equiv 0$ and running cost $\ell_\theta(x, a, b, \omega) \coloneqq \ell(x, a, b, \omega) + \langle f(x, a, b), \theta \rangle$ and $\widehat{\ell}_\theta(x, a, b, \omega) \coloneqq \widehat\ell(x, a, b, \omega) + \langle f(x, a, b), \theta \rangle$, respectively.
	If $\ell = \widehat{\ell}$ on $\left( \M \setminus S_{R, \widehat{R}} \right) \times A \times B$, we have  
    \[
        |u_\theta(t, x) - \widehat{u}_\theta(t, x) | 
        	\leqslant \frac{\widehat{R}-R}{\delta} \| \ell - \widehat{\ell} \|_\infty
		\qquad \hbox{for all $(t, x) \in \cyl$.}
    \]
\end{lemma}

We point out the generality of the previous statement: no continuity conditions on the running costs $\ell,\,\widehat\ell$ are assumed in the above statement, and the functions $u_\theta,\,\widehat u_\theta$ are still well-defined.

\begin{proof}[Proof of Lemma \ref{lem:strip}]
    Fix $t > 0$ and $x, \theta \in \M$. 
    Let $R < \widehat{R}$ be arbitrary in $\R$ and consider the strip $S_{R, \widehat{R}}$.
	Fix controls $\alpha \in \Gamma(t)$ and $b \in \BB(t)$ and consider the solution $y_x \colon [0, t] \to \M$ of the ODE 
	\begin{eqnarray*}
	   	\begin{cases}
	   		\dot y_x(s) = f(\alpha[b](s), b(s))\qquad\hbox{in $[0, t]$} \\
	   		y_x(0) = x.
	   	\end{cases}
	\end{eqnarray*}  
	From the orientation of the game, the map $s \mapsto \langle y_x(s), e\rangle$ is strictly increasing in $[0, t]$. 
	More precisely, 
	\begin{equation}\label{eq strict monotonicity}
		\dfrac{d}{ds} \langle y_x(s), e \rangle 
			= \langle \dot y_x(s), e \rangle 
			= \langle f(\alpha[b](s), b(s)), e \rangle 
			\geqslant \delta
			\qquad \hbox{for all $s \in [0, t]$}.
	\end{equation}
	If $ \langle y_x(0), e \rangle = \langle x, e \rangle \geqslant \widehat{R}$, we derive that the curve $y_x$ always lies in $\M \setminus S_{R, \hat R}$ and the assertion trivially follows since $\ell = \widehat{\ell}$ on $\left( \M \setminus S_{R, \widehat{R}} \right) \times A \times B$. 
	Let us then assume that $\langle y_x(0), e \rangle < \widehat{R}$ and define two exit times $t_1$ and $t_2$ as follows:
    \[
        t_1 \coloneqq \inf \left\{ s \in [0, t] : \langle y_x(s), e \rangle > R \right\}\,,\qquad 
        t_2 \coloneqq \sup \left\{ s  \in [0, t]  : R < \langle y_x(s), e \rangle < \widehat{R} \right\},
    \]
	where we agree that $t_1=t_2=t$ when the sets above are empty.  
	Notice that\ $t_2 - t_1 \leqslant (\widehat{R} - R) / {\delta}$.  
	Indeed, if $t_2 - t_1 > 0$, then $t_1 < t$ and, by continuity of $y_x$ and \eqref{eq strict monotonicity}, we have $\langle y_x(t_1), e \rangle = R$. 
	From \eqref{eq strict monotonicity}, we infer 
	 \begin{align*}
	     \widehat{R} - R 
	     	\geqslant \langle y_x(t_2) - y_x(t_1), e \rangle 
	        = \int_{t_1}^{t_2} \langle f(\alpha[b](s),b(s)), e \rangle \, ds 
	        \geqslant (t_2 - t_1) \delta \,,
	 \end{align*}
	as it was claimed.
	Also, if $t_2 < t$, then, from \eqref{eq strict monotonicity}, we have that $\langle y_x(s), e \rangle > \widehat{R}$ for every $s \in (t_2, t)$. 
	Consider deterministic running costs $\ell, \widehat{\ell} \colon \RR^d \times A \times B \to \RR$ such that $\ell = \widehat{\ell}$ on $\left( \M \setminus S_{R, \widehat{R}} \right) \times A \times B$.
	Then, in view of the previous remarks, we get 
	\begin{align*}
		\Big|
			\int_0^t \Big( \ell_\theta( y_x(t), \alpha(t), \beta(t)) &- \widehat{\ell}_\theta( y_x(t), \alpha(t), \beta(t))\Big) \, ds
		\Big| \\
			&= \Big|
				\int_{t_1}^{t_2} \Big(\ell( y_x(t), \alpha(t), \beta(t)) - \widehat{\ell}( y_x(t), \alpha(t), \beta(t))\Big)\, ds
			\Big|  \\        
			& \leqslant (t_2 - t_1) \|\ell - \widehat{\ell}\|_\infty  
 			\leqslant \frac{\widehat{R} - R}{\delta} \|\ell - \widehat{\ell}\|_\infty \,.
	\end{align*}
	The assertion easily follows from this by arbitrariness of the choice of the control $b \in \BB(t)$ and of the strategy $\alpha \in \Gamma(t)$. 
\end{proof}

In the proof of \Cref{prop:concentration}, we look at the conditional expectation of $u_\theta(t, 0, \cdot)$ given the running cost in a {half-space that contains the origin and whose boundary is a hyperplane orthogonal to $e$. Considering a suitable increasing sequence of such half--spaces, we define the corresponding martingale and, by \Cref{Result: Deterministics payoff dependence on strips of the $T$-horizon value}, observe that it has bounded differences. Then, applying Azuma's inequality, see \Cref{Result: Concentration of martingales}, we conclude.}

\begin{proof}[Proof of Proposition~\ref{prop:concentration}]
	Fix $t > 0$. 
	Recall the long-range independence parameter $\rho \ge 1$.
	Denote $n \coloneqq \lceil t \rceil$ and $C \coloneqq \lceil \| f \|_\infty / \rho \rceil$, where $\lceil \cdot \rceil$ stands for the upper integer part.
	Note that, for all $\alpha\in\Gamma(t)$ and $b \in \BB(t)$, the solution $y_0 \colon [0, t] \to \M$ of the ODE 
	\begin{eqnarray*}
		\begin{cases}
			\dot y_0(s) = f(\alpha[b](s), b(s))
				\qquad\hbox{in $[0,t]$} \\
			y_0(0) = 0
		\end{cases}
	\end{eqnarray*}	
	satisfies that $|y_0(s)| \leqslant \rho\, C \, n$ for all $s \in [0, t]$.
	For $r \in \{1, 2, \ldots, C n \}$, let $\mathcal{F}_r$ be the $\sigma$-algebra generated by the random variables $\{ \ell(x, a, b, \cdot) : a \in A, b \in B,  {x \in \RR^d, }\langle x, e \rangle \leqslant \rho r \}$, and let $\mathcal{F}_0$ be the trivial $\sigma$-algebra. 
	We claim that, for all $0 \leqslant r < C n$, we have that
	\begin{equation} \label{eq_increment}
		\left| \EE [ u_\theta(t, 0, \cdot) | \mathcal{F}_{r + 1} ] - \EE [ u_\theta(t, 0, \cdot) | \mathcal{F}_r ] \right| 
			\leqslant \frac{12 \rho^2}{\delta} \Lip(\ell) \,.
	\end{equation}
	Indeed,  {for $R < \widehat{R}$ in $\R$, define the strip}
    \[
        S_{R, \widehat{R}} \coloneqq \{ x \in \M : \langle x, e\rangle \in [R, \widehat{R}] \} \,.
    \]
     {Fix $0 \leqslant r < C n$, }and consider {$\widehat{\ell}$ defined by  {$\widehat{\ell}(x, a, b, \omega) \coloneqq \ell(x - 3 \rho e, a, b, \omega)$ for $(x, a, b, \omega) \in S_{\rho (r - 1), \rho(r + 2)} \times A \times B \times \Omega$}, and $\widehat{\ell}(x, a, b, \omega) \coloneqq \ell(x, a, b, \omega)$ otherwise.}
	
    {On the one hand}, by \Cref{Result: Deterministics payoff dependence on strips of the $T$-horizon value}, almost surely we have that
	\begin{align}
		\label{Equation: distance between solutions}
		\left| u_\theta(t, 0, \omega) - \widehat{u}_\theta(t, 0, \omega) \right| 
			\leqslant \frac{ {3} \rho}{\delta} \| \ell - \widehat{\ell} \|_\infty 
            \leqslant \frac{ {3} \rho}{\delta} \Lip(\ell) 2 \rho 
			= \frac{ {6} \rho^2}{\delta} \Lip(\ell).
	\end{align}
	where $u_\theta$ (resp. $\widehat{u}_\theta$) is the value function associated via \eqref{def value function} with the running cost $\ell_\theta \coloneqq \ell + \langle f(a, b), \theta \rangle$ (resp. $\widehat\ell_\theta \coloneqq \widehat \ell + \langle f(a, b), \theta \rangle$) and initial datum $g \equiv 0$.
    On the other hand, $\widehat{u}_\theta$ is independent of $\{ \ell(x, a, b, \cdot) : a \in A, b \in B, {x \in \RR^d, }\langle x, e \rangle \in [ \rho r,  \rho (r + 1)] \}$. 
    Indeed, by definition, 
    $\widehat{u}_\theta$ is measurable with respect to the $\sigma$-field generated by $\{ \widehat{\ell}(x, a, b, \cdot) : a \in A, b \in B, x \in \M \}$.
    Moreover, by long-range independence, $\widehat{\ell}(x, a, b, \cdot)$ is independent of $\{ \ell(x, a, b, \cdot) : a \in A, b \in B, {x \in \RR^d, }\langle x, e \rangle \in [ \rho r, \rho (r + 1)] \}$. 
    Therefore, 
    \begin{align}
		\label{Equation: conditional expectations of null solution}
		\EE \left[ \widehat{u}_\theta(t, 0, \cdot) \cond \mathcal{F}_{r + 1} \right]
			=\EE \left[ \widehat{u}_\theta(t, 0, \cdot) \cond \mathcal{F}_{r} \right] \,.
	\end{align}
	Finally, using successively \eqref{Equation: distance between solutions} and \eqref{Equation: conditional expectations of null solution}, we prove the claim in \eqref{eq_increment}:	
	\begin{align*}
		\bigl | \EE [ u_\theta(t, 0, \cdot) | \mathcal{F}_{r + 1}] - \EE[u_\theta(t, 0, \cdot) | \mathcal{F}_r]  \bigr | 
			&\leqslant 
			\bigl | \EE[\widehat{u}_\theta(t, 0, \cdot) | \mathcal{F}_{r + 1}] - \EE[\widehat{u}_\theta(t, 0, \cdot) | \mathcal{F}_r] \bigr | + \frac{12 \rho^2}{\delta} \Lip(\ell) \\
			&= \frac{12 \rho^2}{\delta} \Lip(\ell)  \,.
	\end{align*}
	For $r \leqslant C n$, denote $W_r \coloneqq \EE[ u_\theta(t, 0, \cdot) | \mathcal{F}_r ]$. 
	The process $(W_r)_{r \leqslant C n}$ is a martingale, and $W_0 = \EE[ u_\theta(t, 0, \cdot) ] = U_\theta(t)$. 
	Moreover, since $u_\theta(t, 0, \cdot)$ is $\mathcal{F}_{C n}$-measurable, we have that $W_{C n}(\cdot) = u_\theta(t, 0, \cdot)$. 
	By Azuma's inequality, see \Cref{Result: Concentration of martingales}, applied to the martingale $(W_r)_{r \leqslant C \, n}$ and \eqref{eq_increment}, for all $M \geqslant 0$,
	\begin{align*}
		\m{P}(|u_\theta(t, 0, \cdot) - U_\theta(t)| \geqslant M) 
			= \m{P}(|W_{C n} - W_0| \geqslant M) 
			\leqslant 2 \exp\left(- \frac{M^2}{2 C n \left( \frac{12 \rho^2}{\delta} \Lip(\ell) \right)^2 } \right) \,.
	\end{align*}
	Therefore, there exists a constant $c > 0$, only depending on $\rho$, $\delta$,  $\Lip(\ell)$, and $\left\|f\right\|_\infty$, but not on $\theta$, such that, for all $M > 0$ and $t \geqslant 1$, we have that 
	\begin{align*}
		\m{P}(|u_\theta(t, 0, \cdot) - U_\theta(t)| \geqslant M \sqrt{t})
			&\leqslant \exp\left(- c M^2 \right) \,,
	\end{align*}
	as it was asserted. 
\end{proof}

\section{Proof of the homogenization results} \label{sec:proof}

This section is devoted to the proofs of our homogenization results, namely Theorems~\ref{thm:genhom}, ~\ref{Result: Concentration property} and~\ref{thm:genhom2} and  Corollary~\ref{cor:quantitative estimate}. 
We will denote by $H$ a stationary Hamiltonian belonging to the class $\Ham$, and by $u_\theta(\cdot, \cdot, \omega)$ the solution of the equation \eqref{eq eps hj} with  $\eps = 1$, $H_\theta \coloneqq H(\cdot, \theta + \cdot\,, \omega)$ in place of $H$, and initial condition $u_\theta(0, x, \omega) = 0$ for all $(x, \omega) \in \M \times \Omega$.
\smallskip 

Note that, since the initial condition is zero, the function $u_\theta$ is stationary, i.e., for all $(t, x, \omega) \in \ccyl \times \Omega$ and $y \in \M$ we have $u_\theta(t, x + y, \omega) = u_\theta(t, x, \tau_y \omega)$, $\PP$--almost surely in $\Omega$.
\smallskip

We also recall that $U_\theta(t)$ denotes the expectation of the random variable $u_{\theta}(t, 0, \cdot)$.
\smallskip 

\subsection{Proof of Theorem~\ref{thm:genhom}} 

In this subsection, we will furthermore assume that $H$ belongs to the class $\Hamdg$, so that $u_\theta(t, x, \omega)$ can be represented via \eqref{def value function} with initial datum $g \equiv 0$ and running cost $\ell_\theta(x, a, b, \omega) \coloneqq \ell(x, a, b, \omega) + \langle f(x, a, b), \theta \rangle$. 
This allows us to make use of the results obtained in Section~\ref{sec:concentration}.  

According to Section~\ref{sec:outline}, the proof of Theorem~\ref{thm:genhom} boils down to establishing the following result. 

\begin{proposition} \label{prop convergence tool}
There exists a deterministic function $\overline H \colon \M \to \R$ such that, for every fixed $\theta \in \M$,  we have 
	\begin{equation}\label{claim concentration}
		\limsup_{\epsilon \to 0}\sup_{y\in B_R} |u^\epsilon_\theta(1, y, \omega)+ \overline{H}(\theta)| 
		= 
		\limsup_{t \to +\infty}\sup_{y\in B_{tR}} \left |\frac{u_\theta(t, y, \omega)}{t} + \overline{H}(\theta)\right|
		=
		0
		\quad
		\hbox{for all $R>0$,}
	\end{equation}	
for every $\omega$ in a set $\Omega_\theta$ of probability 1. 
\end{proposition}

Proposition \ref{prop convergence tool} is actually a consequence of the following stronger quantitative result. 

\begin{prop} 
\label{prop Concentration property}
There exists a deterministic function $\overline H \colon \M \to \R$ such that, for every fixed $\theta \in \M$ and $R>0$,  we have 
\begin{equation}\label{claim concentration inequality}
   \m{P}\left( \sup_{y\in B_{tR}} \left |\frac{u_\theta(t, y, \omega)}{t} + \overline{H}(\theta)\right| \geq 
   K  \left(\frac{\ln t}{t}\right)^{1/2} \right)  \leq t^{-2}    
\quad
\hbox{for all $t\geq 2$,}
\end{equation}
for some constant $K$ depending on $R$,\,$|\theta|$,\,$d$,\,$\beta$,\,$\rho,\delta,\, \Lip(\ell)$ and $\left\|f\right\|_\infty$.
\end{prop}

We first show how Proposition~\ref{prop convergence tool} follows from Proposition~\ref{prop Concentration property}. For later use, we isolate the argument in the following lemma. Note that $u_\theta$ satisfies the required hypotheses by Theorem~\ref{teo app Lip estimates} and  the condition $u_\theta(0,\cdot,\omega)=0$ on $\M$.

\begin{lemma}\label{lemma key implication}
Let $u:[0,+\infty)\times\M\times\Omega\to\R$ be a function such that 
\[
|u(t,x,\omega)-u(s,x,\omega)|\leq \kappa |t-s|,
\ \ 
|u(t,x,\omega)|\leq \kappa(1+t)
\quad
\FORALL\ t,s\in [0,+\infty),\ x\in\M\ \AND\ \omega\in\Omega,
\]
for some constant $\kappa>0$.
Let assume that \eqref{claim concentration inequality} holds for some $R>0$ and $\overline H(\theta)\in\R$. Then \eqref{claim concentration} holds $\PP$-almost surely for the same $R$ and $\overline H(\theta)$. 
\end{lemma}

\begin{proof}
Let us set 
\[
X_n(\omega) \coloneqq  \sup_{y\in B_{nR}} \left |\frac{u(n, y, \omega)}{n} + \overline{H}(\theta)\right|
\quad
\AND
\quad 
\alpha_n \coloneqq  K  \left(\frac{\ln n}{n}\right)^{1/2}\qquad \FORALL\ n\in\N.  
\]
We claim that $\limsup_n X_n(\omega)=0$ for $\PP$--almost every $\omega\in\Omega$. Indeed, note that 
\[
\left\{
\limsup\nolimits_n X_n>0
\right\}
\subseteq 
\bigcap_{k=2}^{+\infty}\underbrace{\bigcup_{n=k}^{+\infty} \left\{ X_n>\alpha_n\right\}}_{E_k}. 
\] 
By hypothesis we have  
\[
\PP\left(\limsup\nolimits_n X_n>0\right)
\leq
\lim_k \PP(E_k)\leq \lim_k\sum_{n=k}^{+\infty} \frac{1}{n^2}=0,
\]
as it was claimed. The assertion follows from this by applying the subsequent lemma with $v \coloneqq u(\cdot,\cdot,\omega)$ for $\PP$-almost every $\omega\in\Omega$. 
\end{proof}

\begin{lemma}
Let $v:[0,+\infty)\times\M\to\R$ be a function such that 
\[
|v(t,x)-v(s,x)|\leq \kappa |t-s|,
\quad
|v(t,x)|\leq \kappa(1+t)
\qquad
\FORALL\ t,s\in [0,+\infty)\ \AND\ x\in\M
\]
for some constant $\kappa>0$. For every $a\in\R$ and $R>0$ we have 
\begin{equation}\label{claim limsup}
\limsup_{t\to+\infty} \sup_{x\in B_{tR}} \left|  \frac{v(t,x)}{t}+a\right|
=
\limsup_{n\to+\infty} \sup_{x\in B_{nR}} \left|  \frac{v(n,x)}{n}+a\right|.
\end{equation}
\end{lemma}

\begin{proof}
It suffices to prove \eqref{claim limsup} with $\leq$ in place of $=$, being the reverse inequality obvious. For every $t\geq 1$ we have 
\[
\sup_{x\in B_{tR}} \left|  \frac{v(t,x)}{t}+a\right|
\leq
\sup_{x\in B_{\lceil t\rceil R}} \left|  \frac{v(t,x)}{t}+a\right|
\leq
\sup_{x\in B_{\lceil t\rceil R}} \left|  \frac{v(\lceil t\rceil ,x)}{\lceil t \rceil}+a\right|
+
\sup_{x\in B_{\lceil t\rceil R}} \left|  \frac{v(\lceil t\rceil ,x)}{\lceil t \rceil}-\frac{v(t,x)}{t}  \right|.
\] 
Now 
\[
\sup_{x\in B_{\lceil t\rceil R}} \left|  \frac{v(\lceil t\rceil ,x)}{\lceil t \rceil}-\frac{v(t,x)}{t}  \right|
\leq
\sup_{x\in B_{\lceil t\rceil R}} \left|  \frac{v(\lceil t\rceil ,x)}{\lceil t \rceil}-\frac{v(t,x)}{\lceil t\rceil}  \right|
+
\kappa(1+t)\left| \frac{1}{\lceil t \rceil}-\frac1t  \right|
\leq 
\frac{\kappa}{\lceil t \rceil}
+
\frac{\kappa(1+t)}{t\lceil t \rceil},
\]
so the assertion follows by sending $t\to +\infty$. 
\end{proof}

Let us now proceed to prove Proposition~\ref{prop Concentration property}. To this aim, we start by stating the following 
fact. 

\begin{proposition} 
\label{prop:expectation}
The following statements hold. 
\begin{enumerate}[(i)]
\item
For all $\theta \in \M$, we have that $U_{\theta}(t) / t$ converges as $t$ goes to infinity.\smallskip
   \item 
   Let
 $\overline H \colon \M \to \R$ be defined by 
	\[
		\overline H(\theta) 
			\coloneqq 
			-\lim_{t\to +\infty} \frac{U_\theta(t)}{t}
			\quad
			\hbox{for all $\theta\in\R^d$.}
		\]
For every fixed $\theta\in\R^d$, there exists a constant $K$, depending on $|\theta|$,\,$d$,\,$\beta$,\,$\rho,\delta,\, \Lip(\ell)$ and $\left\|f\right\|_\infty$, such that
    \begin{equation*}
    \bluee
    \left|\frac{U_{\theta}(t)}{t}+\overline H(\theta)\right| \leq K 
    \left(\frac{\ln t}{t}\right)^{1/2}
    \quad
    \hbox{for all $t \geq 2$.}
    \end{equation*}
   \end{enumerate}
\end{proposition}

The proof of Proposition~\ref{prop:expectation} consists in showing that $U_\theta(t)$, satisfies an approximated subadditive inequality. 
We postpone that proof and first explain how to use this fact to establish \Cref{prop Concentration property}. In doing so, we also rely crucially on \Cref{prop:concentration}.

\begin{proof}[Proof of \Cref{prop Concentration property}]
	According to Proposition~\ref{prop:expectation}, we define $\overline H \colon \M \to \R$ by setting 
	\[
		\overline H(\theta) 
			\coloneqq 
			-\lim_{t\to +\infty} \frac{U_\theta(t)}{t}
			=
			-\lim_{t\to +\infty} \frac{\EE[u_{\theta}(t, 0, \cdot)]}{t}
			\qquad
			\hbox{for all $\theta\in\M$.}
	\]
	Consider $n \in \NN$ and a discretization $\bluee Z_n \coloneqq \frac{1}{n^2}\Z^d\cap   B_{nR}$  of $ B_{nR}$ consisting of at most $(2n^3R)^{d}$ points and such that any point of $B_{nR}$ is $(\sqrt{d} / n^{2})$-close to $Z_n$. 
	Consider $M > 0$ arbitrary.
	Then, for $n \in \NN$, we focus on the event
	\[
		\sup_{z \in Z_n} \left| \frac{u_\theta(n, z,\omega)}{n} + \overline{H}(\theta) \right| \geqslant M \,.
	\]
    
	Using first the union bound, then \Cref{prop:expectation}-(ii), and last the stationarity of $u_\theta$, we have that, for $n$ large enough,
	\begin{align*}
		\PP \left( \sup_{z \in Z_n} \left | \frac{u_\theta(n, z,\cdot)}{n} + \overline{H}(\theta) \right | \geqslant  M+K \left(\frac{\ln n}{n}\right)^{1/2} \right) 
			&\leqslant \sum_{z \in Z_n} \PP \left( \left| \frac{u_\theta(n, z,\cdot)}{n} + \overline{H}(\theta) \right| \geqslant M+K \left(\frac{\ln n}{n}\right)^{1/2} \right) \\
			&\leqslant \sum_{z \in Z_n} \PP \left( \left| \frac{u_\theta(n, z,\cdot)}{n} - \frac{U_{\theta}(n)}{n} \right| \geqslant M \right) \\
			&= |Z_n|\, \PP \left( \left| \frac{u_\theta(n, 0,\cdot)}{n} - \frac{U_{\theta}(n)}{n} \right| \geqslant M \right).
	\end{align*} 
	Note that, by the choice of $Z_n$ and \Cref{prop:concentration}, 
	\[
		|Z_n| \, \PP \left( \left| \frac{u_\theta(n, 0,\cdot)}{n} - \frac{U_{\theta}(n)}{n} \right| \geqslant M \right)
			\leqslant ({\bluee 2}n^3R)^{d} \, \exp \left( -c M^2 n \right) \,,
	\]
    and combining with the previous inequality, we get
    \begin{align} \label{eq:discrete_con}
    \PP \left( \sup_{z \in Z_n} \left | \frac{u_\theta(n, z,\cdot)}{n} + \overline{H}(\theta) \right | \geqslant  M+K \left(\frac{\ln n}{n}\right)^{1/2} \right) \leq ({\bluee 2}n^3R)^{d} \, \exp \left( -c M^2 n \right).
    \end{align} 
We deduce that there exists a constant $K'>K$ large enough, depending on $R$,\,$|\theta|$,\,$d$,\,$\beta$,\,$\rho,\delta$,\,$\Lip(\ell)$ and $\left\|f\right\|_\infty$, such that
     \begin{align} \label{eq:discrete_con2}
    \PP \left( \sup_{z \in Z_n} \left | \frac{u_\theta(n, z,\cdot)}{n} + \overline{H}(\theta) \right | \geqslant  M+K' \big(\ln(n)/n\big)^{\frac{1}{2}} \right) \leq \, \exp \left( -c M^2 n \right)
    \quad
    \hbox{for all $n\geq 2$}.
    \end{align} 
    By applying Theorem~\ref{teo app Lip estimates} with $H_\theta$ in place of $H$ we derive that $u_\theta(t, \cdot, \omega)$ is 
    $\bluee \beta t$--Lipschitz in $\M$ and $u_\theta(\cdot, x, \omega)$ is $\bluee\beta(1+ |\theta|)$--Lipschitz in $[0,+\infty)$.\footnote{Since $\bluee g=0$ and the Hamiltonian $H_\theta$ satisfies (H1$^*$)-(H3$^*$) with 
	$\beta_1 \coloneqq \beta(1+|\theta|)$ and $\beta_2=\beta_3 \coloneqq \beta$.}  
	We infer 
	\begin{equation}\label{eq linear growth u theta}
		|u_\theta(t, x, \omega)| 
			=
			|u_\theta(t, x, \omega)-u_\theta(0, x, \omega)| 
			\leqslant 
			{\bluee \beta t (1 + |\theta|)}
			\qquad
			\hbox{for all $(t, x, \omega) \in \ccyl \times \Omega$.}
	\end{equation}
	The latter implies
	\begin{align*}
		\sup_{y\in   B_{tR}} \left| \frac{u_\theta(t, y,\omega)}{t} + \overline{H}(\theta) \right|
			&\leqslant \sup_{y\in   B_{\lceil t\rceil R}} \left| \frac{u_\theta(t, y, \omega)}{t} + \overline{H}(\theta) \right|
			\\
			&\leqslant \sup_{y\in   B_{\lceil t\rceil R}} \left| \frac{u_\theta(\lceil t\rceil, y, \omega)}{t} + \overline{H}(\theta) \right| + 
			\frac{\bluee \beta(1+|\theta|)}{t}
			\\
			&\leqslant \sup_{y\in   B_{\lceil t\rceil R}} \left| \frac{u_\theta(\lceil t\rceil, y, \omega)}{\lceil t\rceil} + \overline{H}(\theta) \right| + 2\,\frac{\bluee \beta(1+|\theta|)}{t}
			\\
			&\leqslant \sup_{z \in Z_{\lceil t\rceil}} \left| \frac{u_\theta(\lceil t\rceil, z,\omega)}{\lceil t\rceil} + \overline{H}(\theta) \right| + 
			{\beta\,\frac{\sqrt{d}}{\lceil t \rceil^2} }+ 2\, \frac{\bluee \beta(1 + |\theta|)}{t}
	\end{align*}
In view of \eqref{eq:discrete_con2} and of the fact that $1/t=o\big(\log(t)^{1/2}t^{-1/2}\big)$, we deduce that there exists a constant $K''$, depending on $R$,\,$|\theta|$,\,$d$,\,$\beta$,\,$\rho,\delta$,\,$\Lip(\ell)$ and $\left\|f\right\|_\infty$, such that, for all $M>0$, 
    \begin{eqnarray*}
    \PP \left( \sup_{y\in   B_{tR}} \left| \frac{u_\theta(t, y,\omega)}{t} + \overline{H}(\theta) \right| \geqslant  M+K'' \left(\frac{\ln t}{t}\right)^{1/2} \right) 
    \leq \exp(-c M^2 t)
    \qquad
    \hbox{for all $t \geq 2$.}
        \end{eqnarray*}
        Taking $M \coloneqq (\ln t/t)^{1/2} \sqrt{2}{c}^{-1/2}$, we get the result.
\end{proof}

Let us turn back to the proof of \Cref{prop:expectation}. 

\begin{proof}[Proof of {\Cref{prop:expectation}}]
	Fix $\theta \in \M$ and $t > 0$. 
	Denote $\mathbf{B}(t) \coloneqq \overline B(0, t \left\| f \right\|_\infty)$. 
	Note that, for all $\alpha \in \Gamma(t)$ and $b \in \BB(t)$, the solution of the ODE 
	\begin{eqnarray*}
		\begin{cases}
			\dot y_0(s) = f(\alpha[b](s),b(s))\qquad\hbox{in $[0, t]$} \\
			y_0(0) = 0
		\end{cases}
	\end{eqnarray*}
	is such that, for all $s \in [0, t]$, we have that $y_0(s) \in \mathbf{B}(t)$.
	 From Theorem~\ref{teo app Lip estimates} with $H_\theta$ in place of $H$ we derive that $u_\theta(t, \cdot, \omega)$ is 
    $\bluee \beta t$--Lipschitz in $\M$ and $u_\theta(\cdot, x, \omega)$ is $\bluee\beta(1+ |\theta|)$--Lipschitz in $[0,+\infty)$.
	We will discretize $\mathbf{B}(t)$ accordingly.
	Consider a finite set $Z$ of size $\bluee\left\lceil 2 \left\| f \right\|_\infty\,t\,(\beta t)  \right\rceil^{d}$ such that any point of $\mathbf{B}(t)$ is $(\beta t )^{-1}$--close to a point in $Z$. 
	Using the Lipschitz property of $u_\theta(t, \cdot, \omega)$, the union bound, the fact that variables $u_{\theta}(t, x, \omega)$ and $u_{\theta}(t, 0, \omega)$ have the same distribution by stationarity and the concentration proven in \Cref{prop:concentration}, we get that, for all $M > 0$,
	\begin{align*}
		\m{P} \left( \exists x \in \mathbf{B}(t), |u_{\theta}(t,x, \cdot) - U_\theta(t)| \geqslant M \right) 
			&\leqslant \m{P}  \left( {\bluee \exists z \in Z, |u_{\theta}(t,z,\cdot)} - U_\theta(t)| \geqslant M - 1 \right) \\
			&\leqslant \sum_{z \in Z}  \m{P} \left( |u_{\theta}(t, z, \cdot) - U_\theta(t)| \geqslant M - 1 \right) \\
			&=  \sum_{z \in Z} \m{P} \left( |u_{\theta}(t, 0, \cdot) - U_\theta(t)| \geqslant M - 1 \right)  \\
			&\leqslant 
			{\bluee\left\lceil 2 \left\| f \right\|_\infty\,\beta t ^2  \right\rceil^{d}}
			 \exp\left(-c (M - 1)^2 / t \right),
	\end{align*}
	where $c$ is a constant depending on $\rho, \delta, \Lip(\ell)$, and $\left\| f \right\|_\infty$, but not on $\theta$.

	Taking $M_t {\bluee \coloneqq \big(\ln (\lceil 2 \left\| f \right\|_\infty\,\beta t^2 \rceil^{d} t)\big)^{1/2}c^{-1/2} t^{1/2} + 1}$ in place of $M$ in the above inequality, we get
	\begin{align*}
		\m{P}(\exists x \in \mathbf{B}(t), |u_{\theta}(t, x, \cdot) - U_\theta(t)| \geqslant M_t)
			& \leq {\bluee t^{-2}}
	\end{align*}
{\bluee In particular,}
	\begin{equation}
		\label{Equation: Low probability for low value}
		\m{P}\left(\inf_{x \in \mathbf{B}(t)}{u_{\theta}(t,x,\cdot)} \leqslant  U_\theta(t) - M_t \right) 
			\leqslant {\bluee t^{-2}},
			\qquad
		\m{P}\left(\sup_{x \in \mathbf{B}(t)}{u_{\theta}(t,x,\cdot)} \geqslant  M_t+U_\theta(t) \right) 
			\leqslant {\bluee t^{-2}},	
	\end{equation}
We claim that there exists a constant $\widehat{K}$ depending on {\bluee $|\theta|$,\,$d$,\,$\beta$,\,$\rho,\delta,\, \Lip(\ell)$ and $\left\|f\right\|_\infty$} such that	
	\begin{equation}
	\label{eq:minv}
		\EE \left[\inf_{x \in \mathbf{B}(t)}{u_{\theta}(t, x, \cdot)}\right] 
			\geqslant U_\theta(t) - \widehat{K}  \left( t\ln t\right)^{1/2}\!\!\!\!\!\!\!\!,
			\qquad
		\EE \left[\sup_{x \in \mathbf{B}(t)}{u_{\theta}(t, x, \cdot)}\right] 
			\leqslant U_\theta(t) + \widehat{K}  \left(t\,{\ln t}\right)^{1/2}
	\end{equation}
for all $t\geq 1$. 
Let us prove the first inequality. We recall that $\bluee |u_\theta(t, x, \omega)| \leqslant \beta t (1 + |\theta|)$, see \eqref{eq linear growth u theta}, in particular $\bluee |U_\theta(t)| \leqslant \beta t (1 + |\theta|)$.
	Then,
	\begin{align*} 
		\EE \left[ \inf_{x \in \mathbf{B}(t)}{u_{\theta}(t, x, \cdot)} \right] 
			&\geqslant \m{P}\left(\inf_{x \in \mathbf{B}(t)}{u_{\theta}(t, x, \cdot)} \leqslant  U_\theta(t) - M_t 					\right) 
				{\bluee \left( -\beta t(1+|\theta|) \right )} 
\\
				&\qquad + \m{P}\left(\inf_{x \in \mathbf{B}(t)}{u_{\theta}(t, x, \cdot)} >  U_\theta(t) - M_t 					\right) \left( U_\theta(t) - M_t \right). 
	\end{align*}
{\bluee
Now we use the identity 
$\m{P}\left(\inf_{x \in \mathbf{B}(t)}{u_{\theta}(t,x,\cdot)} >  U_\theta(t) - M_t \right) 
=
1-\m{P}\left(\inf_{x \in \mathbf{B}(t)}{u_{\theta}(t,x,\cdot)} \leqslant  U_\theta(t) - M_t \right)$  
and the fact, observed above, that $-U_\theta\geq-\beta t (1 + |\theta|)$. We get			
	\begin{align*} 
	\EE \left[ \inf_{x \in \mathbf{B}(t)}{u_{\theta}(t, x, \cdot)} \right] 				
			&\geq \m{P}\left(\inf_{x \in \mathbf{B}(t)}{u_{\theta}(t, x, \cdot)} \leqslant  U_\theta(t) - M_t \right) \left( - {\bluee \beta t(1+|\theta|)} - U_\theta(t) + M_t \right) + U_\theta(t) - M_t
			\\
			&\geqslant U_\theta(t) - M_t - \m{P}\left(\inf_{x \in \mathbf{B}(t)}{u_{\theta}(t, x, \cdot)} \leqslant  U_\theta(t) - M_t \right)2\beta t(1+|\theta|).
	\end{align*}
}
In view of the first inequality in \eqref{Equation: Low probability for low value}, we get the first inequality in \eqref{eq:minv}. 

To prove the second inequality in \eqref{eq:minv}, we argue analogously. We have 
	\begin{align*} 
		\EE \left[ \sup_{x \in \mathbf{B}(t)}{u_{\theta}(t, x, \cdot)} \right] 
			&\leqslant \m{P}\left(\sup_{x \in \mathbf{B}(t)}{u_{\theta}(t, x, \cdot)} <  U_\theta(t) + M_t 					\right) 
			\left( U_\theta(t) + M_t\right) 	
\\
				&\qquad + \m{P}\left(\sup_{x \in \mathbf{B}(t)}{u_{\theta}(t, x, \cdot)} \geq U_\theta(t) + M_t 					\right)  \beta t(1+|\theta|)\\
				&\leq
			U_\theta(t) + M_t
			+
			\m{P}\left(\sup_{x \in \mathbf{B}(t)}{u_{\theta}(t, x, \cdot)} \geq U_\theta(t) + M_t 					\right)  2\beta t(1+|\theta|),
	\end{align*}
and the second inequality in  \eqref{eq:minv} follows in view of the second inequality in \eqref{Equation: Low probability for low value}. 

	We will now show that the sequence $(-U_\theta(n))_{n \in \NN}$ is almost subadditive and therefore it has a limit.
	Let $n \geqslant 1$ and $\bluee m, n\in\N$. 
	According to Theorems~\ref{teo appendix representation} and~\ref{teo dynamic programming principle}, we have the following formulae:
	\begin{equation} 
	\label{eq:value}
		u_\theta(m,0,\omega) 
			= \sup_{\alpha\in\Gamma(m)} \; \inf_{b\in\BB(m)} \left\{ 
				\int_0^m \ell_\theta(y_0(s), \alpha[b](s), b(s), \omega) \,ds 
			\right\} \,,
	\end{equation}
	and
	\begin{equation} 
	\label{eq:DPm}
		u_\theta(m+n,0,\omega)
			= \sup_{\alpha\in\Gamma(m)} \; \inf_{b\in\BB(m)} \left\{
				\int_0^{m} \ell_\theta(y_0(s), \alpha[b](s), b(s), \omega) \,ds + u_\theta(n, y_0(m), \omega) 
			\right\} \,.
	\end{equation}
	By \ref{eq:value} and \ref{eq:DPm}, we have
	\begin{equation*} 
\bluee		u_\theta(m + n, 0, \omega) 
			\geqslant u_\theta(m, 0, \omega) 
				+ \inf_{x \in \mathbf{B}(m)} u_\theta(n, x, \omega) \,.
	\end{equation*}
	By taking expectation and by  using the first inequality in \eqref{eq:minv}, we get
	\begin{eqnarray}\label{inequality sub}
		U_{\theta}(m + n) \geqslant U_\theta(m) + U_\theta(n) - \widehat{K}\, {\bluee\big(n\ln(n)\big)^{1/2}}
		\qquad
		\hbox{ for all $n\in\N$.}
	\end{eqnarray}
	Set $a_n \coloneqq -U_\theta(n)$ for all $n \in \N$,  and $z(h) \coloneqq \widehat{K} \big(h\ln(h)\big)^{1/2}$ for all $h \geqslant 1$. 
	By the previous inequality, the sequence $(a_n)$ is \emph{subadditive with an error term $z$}, i.e.,  
	\begin{equation} \label{eq:sub}
		a_{m+n} \leqslant a_m + a_n + z(m + n)
		\qquad
		\hbox{for all $m,n \in \N$.} 
	\end{equation}
	Note that $z$ is non-negative, non-decreasing  and $\int_1^{+\infty} z(h) / h^2\,dh < +\infty$. Furthermore, by {\bluee Theorem~\ref{teo app Lip estimates}, the function $U_\theta$ is  $\beta(1+|\theta|)$--Lipschitz} in $[0,+\infty)$, in particular $a_n/n$ is bounded since $U_\theta(0) = 0$. 
	By \cite[Theorem 23, page 162]{BE52}, the sequence $(a_n / n)_{n \in \NN}$ converges to a limit, {\bluee that we shall call $\overline H(\theta)$.}  {\bluee We want to estimate the rate of convergence. To this aim, we remark that inequality {\eqref{eq:sub}} gives by induction that, for all $p,n\in\N$,  }
\begin{eqnarray*}
\bluee 
a_{2^{p} n} \leq 
2^{p} a_n+2^p\widehat{K}\sum_{k=1}^p 2^{-k/2}\big(n\ln(2^kn)\big)^{1/2},
\end{eqnarray*}
{\bluee from which we derive
\[
a_{2^{p} n} 
\leq 
2^{p} a_n+2^p A (n\ln(n))^{1/2}\qquad\FORALL\ n\geq 2, 
\]
with $A \coloneqq \widehat{K}\sum_{k=1}^{+\infty} 2^{-k/2}\sqrt{k+1}$.}
%
%
By dividing the above inequality by $2^p n$ and sending $p\to +\infty$, we end up with 
\begin{eqnarray}\label{eq1:half rate}
\overline H(\theta) \leq a_n/n+A \ln(n)^{1/2} n^{-1/2}\qquad\FORALL\ n\geq 2.
\end{eqnarray}

{\bluee
Let us now estimate from below the term $\overline H(\theta)-a_n/n$.  In view of \eqref{eq:DPm}, for every fixed  $\delta>0$ we can choose a strategy  $\alpha\in\Gamma(m)$ such that 
\[
u_\theta(m+n,0,\omega)-\delta
			\leq 
			\inf_{b\in\BB(m)} \left\{
				\int_0^{m} \ell_\theta(y_0(s), \alpha[b](s), b(s), \omega) \,ds + u_\theta(n, y_0(m), \omega) 
			\right\} \,.
\]   
In view of \eqref{eq:value} we get 
\begin{equation*} 
		u_\theta(m + n, 0, \omega) -\delta
			\leqslant u_\theta(m, 0, \omega) 
				+ \sup_{x \in \mathbf{B}(m)} u_\theta(n, x, \omega) \,.
	\end{equation*}
 	By taking expectation, by the arbitrariness of $\delta > 0$ and by using the second inequality in \eqref{eq:minv}, we get
	\begin{eqnarray}\label{inequality2 sub}
		U_{\theta}(m + n) \leqslant U_\theta(m) + U_\theta(n) + \widehat{K}\, {\bluee\big(n\ln(n)\big)^{1/2}}
		\qquad
		\hbox{ for all $n\in\N$.}
	\end{eqnarray}
By setting, as above,	 $a_n \coloneqq -U_\theta(n)$ for all $n \in \N$,  we obtain
\begin{equation} \label{eq:sub2}
		a_m + a_n
		\leq 
		a_{m+n}+ \widehat{K}\, {\bluee\big(n\ln(n)\big)^{1/2}}
		\qquad
		\hbox{for all $m,n \in \N$.} 
	\end{equation}
By induction, we derive, for all $p,n\in\N$, 
\begin{eqnarray*}
2^{p} a_n
\leq
a_{2^{p} n}
+
2^p\widehat{K}\sum_{k=1}^p 2^{-k/2}\big(n\ln(2^kn)\big)^{1/2},
\end{eqnarray*}
from which we derive
\[
2^{p} a_n
\leq
a_{2^{p} n} 
+2^p A (n\ln(n))^{1/2}\qquad\FORALL\ n\geq 2, 
\]
with  $A \coloneqq \widehat{K}\sum_{k=1}^{+\infty} 2^{-k/2}\sqrt{k+1}$.
By dividing the above inequality by $2^p n$ and sending $p\to +\infty$, we end up with 
\begin{eqnarray}\label{eq2:half rate}
a_n/n
\leq 
\overline H(\theta) +A \ln(n)^{1/2} n^{-1/2}\qquad\FORALL\ n\geq 2.
\end{eqnarray}
By putting together inequalities \eqref{eq1:half rate} and \eqref{eq2:half rate} we finally get
\begin{eqnarray}\label{eq:full rate}
|U_{\theta}(n)/n+\overline{H}(\theta)| \leq A \ln(n)^{1/2} n^{-1/2}\qquad\FORALL\ n\geq 2.
\end{eqnarray} 
The assertion follows by the Lipschitz character of the function $U_\theta$.}
\end{proof}

\begin{remark} \label{remark:sub}
{\bluee We remark for further use what we have actually shown with \eqref{inequality sub} and \eqref{eq:full rate}: there exist a constants $\widehat{K}$ and $A$, only depending on $|\theta|$,\,$d$,\,$\beta$,\,$\rho,\delta,\, \Lip(\ell)$ and $\left\| f \right\|_\infty$, such that 
\begin{eqnarray*}
	U_{\theta}(m + n) \geqslant U_\theta(m) + U_\theta(n) - \widehat{K}{\bluee\big(n\ln(n)\big)^{1/2}}	\qquad
	\hbox{ for all $m,n\in\N$}
\end{eqnarray*}
and
\[
\left|U_\theta(n)-\lim_n\frac{U_\theta(n)}{n}  \right|
\leq 
A \ln(n)^{1/2} n^{-1/2}\qquad\FORALL\ n\geq 2.
\]
}
\end{remark}
{\bluee
\subsection{Proof of \Cref{Result: Concentration property}}\label{sec: Concentration property} Let us fix $\theta\in\M$. We first remark that 
$u^\eps_\theta(t,x,\omega)=\tilde u^\eps_\theta(t,x,\omega)-\langle\theta,x\rangle$ and, by rescaling, 
\begin{equation}\label{eq:rescaling again}
u^\epsilon_\theta(t, x,\omega) 
= 
\eps  u_\theta(t / \epsilon, x / \epsilon,\omega)			
\qquad\text{for all $(t ,x,\omega)\in (0,+\infty)\times\R^d\times\Omega$}.
\end{equation}
Let us fix $T>0$ and $R>0$. Proposition \ref{prop Concentration property} with $t \coloneqq 1/\eps$ yields, in view of  \eqref{eq:rescaling again}, 
\begin{equation}\label{claim concentration inequality epsilon}
   \m{P}\left( \sup_{x\in B_{R}} \left |{u^\eps_\theta(1, x, \omega)} + \overline{H}(\theta)\right| \geq 
  \widehat K  \left({-\eps \ln \eps}\right)^{1/2} \right)  \leq \eps^{2}    
\quad
\hbox{for all $\eps\leq 1/2$,}
\end{equation}
for some constant $\widehat K$ depending on $R$,\,$|\theta|$,\,$d$,\,$\beta$,\,$\rho,\delta,\, \Lip(\ell)$ and $\left\|f\right\|_\infty$. From Theorem~\ref{teo app Lip estimates} with $H(x/\eps,\theta+p,\omega)$ in place of $H$, we derive that $u^\eps_\theta(\cdot,x,\omega)$ is $\beta(1+|\theta|)$--Lipschitz in $[0,+\infty)$. In particular, 
\begin{equation}\label{eq:trivial inequality}
|u^\eps_\theta(t,x,\omega)+t\overline H(\theta)|
\leq
|u^\eps_\theta(1,x,\omega)+\overline H(\theta)|
+
2\beta(1+|\theta|)T,
\end{equation}
where we also used the fact that $\overline H$ enjoys (H1). Let us choose $K>0$ large enough so that 
\[
K(-\eps\ln \eps)^{1/2}
\geq
\widehat K(-\eps\ln \eps)^{1/2}
+
2\beta(1+|\theta|)T
\qquad
\FORALL\ \eps\leq 1/2.
\]
The choice of such a constant $K$ clearly depends on $T,\,|\theta|,\,\beta$, and on $R$,\,$d$,\,$\rho$,\,$\delta$,\,$\Lip(\ell)$,\,$\left\|f\right\|_\infty$ through $\widehat K$. In view of \eqref{eq:trivial inequality} and \eqref{claim concentration inequality epsilon} we derive 
\[
   \m{P}\left( \sup_{[0,T]\times B_{R}} \left |{u^\eps_\theta(t, x, \omega)} + t\overline{H}(\theta)\right| \geq 
   K  \left(-\eps {\ln \eps}\right)^{1/2} \right)  \leq \eps^{2}    
\quad
\hbox{for all $\eps\leq 1/2$,}
\]
as it was to be shown. \qed

\subsection{Proof of Corollary \ref{cor:quantitative estimate}} For every fixed $\theta\in\M$, $\eps>0$ and $n\in\N$, let us denote by $\tilde u^\eps_{n_\theta}$ the solution of equation \eqref{eq eps hj} satisfying $\tilde u^\eps_{n_\theta}(0,x, \omega) = \langle \theta, x\rangle$ for all $(x,\omega)\in\M\times\Omega$. 
According to Theorem~\ref{Result: Concentration property}, for every fixed $T>0$ and $R>0$ there exists a constant $K$ depending on $R$,\,$T$\,$|\theta|$,\,$d$,\,$C$,\,$\delta,\, \Lip(\ell)$ such that, for each $n\in\N$,
\begin{equation}\label{eq:concentration estimate2}
   \m{P}\left( \sup_{[0,T]\times B_{R}} \left |{\tilde u^\eps_{n_\theta}(t, x, \omega)}-\langle \theta, x\rangle + t\overline H_n(\theta)\right| \geq 
   K  \left(-\eps {\ln \eps}\right)^{1/2} \right)  \leq \eps^{2}    
\quad
\hbox{for all $\eps\leq 1/2$.}
\end{equation}
According to Proposition~\ref{app prop dense}, each $\overline H_n$ satisfies conditions (H1)--(H2) with $\beta \coloneqq \sup_n\beta_n$. We infer that $(\overline H_n)_n$ is a sequence of equi--Lipschitz and locally equi--bounded functions on $\M$, hence pre-compact in $\CC(\M)$ by the Ascoli--Arzel\`a Theorem. Let $G_1,G_2$ be a pair of accumulations points in $\CC(\M)$ for the sequence $(\overline H_n)_n$, i.e., there exists two diverging sequences $(n^1_k)_k,\,(n^2_k)_k$ such that $\lim_k \overline H_{n^i_k}=G_i$ for $i\in\{1,2\}$. Since $H_n(\cdot,\cdot,\omega)\ucv H(\cdot,\cdot,\omega)$ in $\TM$ for every $\omega\in\Omega$, by the stability property of viscosity solutions we infer that $\tilde u^\eps_{n_\theta}(\cdot,\cdot,\omega)\ucv \tilde u^\eps_\theta(\cdot,\cdot,\omega)$ in $\ccyl$  for every $\omega\in\Omega$, where $\tilde u^\eps_\theta$ denotes the solution of equation \eqref{eq eps hj} satisfying $\tilde u^\epsilon(0,x, \omega) = \langle \theta, x\rangle$ for all $(x,\omega)\in\M\times\Omega$. By passing to the limit in \eqref{eq:concentration estimate2} along the subsequence $(n^i_k)_k$ we infer 
\begin{equation}\label{eq:concentration estimate3}
   \m{P}\left( \sup_{[0,T]\times B_{R}} \left |{\tilde u^\eps_\theta(t, x, \omega)}-\langle \theta, x\rangle + t\overline G_i(\theta)\right| \geq 
   K  \left(-\eps {\ln \eps}\right)^{1/2} \right)  \leq \eps^{2}    
\quad
\hbox{for all $\eps\leq 1/2$.}
\end{equation}
By triangular inequality, this implies in particular
\[
   \m{P}\left( T|G_1(\theta)-G_2(\theta)| \geq 
   K  \left(-\eps {\ln \eps}\right)^{1/2} \right)  \leq \eps^{2}    
\quad
\hbox{for all $\eps\leq 1/2$,}
\]
which implies that $G_1(\theta)=G_2(\theta)$. By arbitrariness of $\theta\in\R^d$, we conclude that the sequence $(\overline H_n)_n$ converges in $\CC(\M)$ to a function $\overline H$, which satisfies (H1)--(H2) with $\beta \coloneqq \sup_n\beta_n$. This proves items (i) and (ii). 

Arguing as we did at the beginning of Section \ref{sec: Concentration property}, we see that \eqref{eq:concentration estimate3} implies inequality \eqref{claim concentration inequality} in the statement of  Proposition \ref{prop Concentration property}. The fact that the HJ equation \eqref{eq eps hj} homogenizes with effective Hamiltonian $\overline H$ follows from this in view of Lemma \ref{lemma key implication}, of the arbitrariness of the choices of $\theta\in\M$ and $R>0$ and of the reduction arguments presented in Section \ref{sec:outline}.. \qed
}

\subsection{Proof of \Cref{thm:genhom2}}
Let $G$ be a stationary Hamiltonian satisfying conditions (G1) and (H1)--(H3), for some $\beta > 0$. 
According to \cite[Lemma 5.1]{EvSo84}, for every fixed $R > 0$ the Hamiltonian $G$ can be represented as 
\begin{equation*}
	G(x, p, \omega)
		= \max_{b \in \bluee B(R)} \min_{a \in A } \left\{-\ell(x, a, b, \omega) - \langle \hat f(a), p \rangle \right\}
		\qquad
		\hbox{for all $(x,p,\omega)\in\M\times\overline B_R\times\Omega$},
\end{equation*}
where $A \coloneqq \overline B_1$, $\bluee B(R) \coloneqq \overline B_R$, $\ell(x, a, b) \coloneqq -G(x, b, \omega) + \beta \langle a, b \rangle$ and $\hat f(a) \coloneqq -\beta a$. 
We derive that 
\begin{equation}\label{eq representation H}
	H(x, p, \omega)
		=\max_{b \in\bluee B(R)} \min_{a \in A} \left\{ -\ell(x, a, b, \omega) - \langle f(a), p \rangle \right\}
		\qquad
		\hbox{for all $(x, p, \omega) \in \M \times \overline{B}_R \times \Omega$},
\end{equation}
where $f(a) \coloneqq \pi^\top( \hat f(a)) + v$ and $\pi^\top$ denotes the transpose of the linear map $\pi$. 
It is clear that $f$ satisfies condition $(f)$ with $e \coloneqq v / |v|$ and $\delta \coloneqq |v|$, and $\ell$ satisfies conditions ($\ell_1$)--($\ell_4$). 
Furthermore, $\bluee \|f\|_\infty \leqslant \beta+\delta$ and $\Lip(\ell) \leqslant \beta$, independently on the choice of $R > 0$ in the definition of the set $B(R) \coloneqq \overline B_R$.\smallskip

{\bluee Now let us fix $\theta\in\M$ and denote by $u_\theta$ the solution of equation \eqref{eq eps hj} with $\eps = 1$ and with $H(\cdot, \theta + \cdot\,, \cdot)$ in place of $H$, subject to the initial condition $u_\theta(0, x, \omega) = 0$ for all $(x, \omega) \in \M \times \Omega$.} For every fixed $s\geqslant 1$, define  
\begin{equation}\label{eq def H^s}
	H^s(x,p,\omega) \coloneqq \max_{b \in \bluee B({\beta s})} \min_{a \in A } \left\{-\ell(x, a, b, \omega) - \langle f(a), p \rangle \right\},
	\qquad
	\hbox{$(x, p, \omega) \in \M \times \M \times \Omega$}.
\end{equation}
Let us denote by $u_\theta^s(t, x, \omega)$ the solution of equation \eqref{eq eps hj} with $\eps = 1$ and with $H^s(\cdot, \theta + \cdot\,, \cdot)$ in place of $H$, subject to the initial condition $u_\theta^s(0, x, \omega) = 0$ for all $(x, \omega) \in \M \times \Omega$. 
{\bluee From \Cref{teo app Lip estimates} (with $g=0$, $\beta_1 \coloneqq \beta(1+|\theta|)$ and $\beta_2=\beta_3 \coloneqq \beta$) we know that both  
$\| D_x u\|_{L^\infty([0,s]\times\M)}$ and $\| D_x u^s\|_{L^\infty([0,s]\times\M)}$ are bounded from above by $\beta s$. Since $H(x, p, \omega) = H^s(x, p, \omega)$ on $\M \times \overline{B}_{\beta s} \times \Omega$, we infer that  $u^s_\theta(t, \cdot, \cdot) \equiv u_\theta(t, \cdot, \cdot)$ for all $0 \leqslant t \leqslant s$.  
}
{\bluee By applying Proposition~\ref{prop:concentration} to $u^t_\theta$, we derive that there exists} a constant $c > 0$, only depending on  {\bluee $|\theta|$,\,$d$,\,$\beta$,\,$\rho,\delta=|v|$ (notice that $\Lip(\ell)$ and $\left\| f \right\|_\infty$ are bounded above by $\beta+\delta$),} such that, for all $M > 0$ and $t \geqslant 1$,
\begin{equation*}
	\m{P} \left( |u_{\theta}(t, 0,\cdot) - U_\theta(t)| \geqslant M \sqrt{t} \right) 
		=
		\m{P} \left( |u^t_{\theta}(t, 0,\cdot) - U^t_\theta(t)| \geqslant M \sqrt{t} \right) 
		\leqslant 
		\exp \left( -c M^2 \right) \,.
\end{equation*}
Moreover, as underlined in \Cref{remark:sub}, there exists a constant $\widehat{K}$, only depending on 
{\bluee $|\theta|$,\,$d$,\,$\beta$,\,$\rho,\delta=|v|$ (notice that $\Lip(\ell)$ and $\left\| f \right\|_\infty$ are bounded above by $\beta+\delta$),}  such that
\begin{eqnarray*}
	U^s_{\theta}(m + n) \geqslant U^s_\theta(m) + U^s_\theta(n) - \widehat{K} {\bluee\big(n\ln(n)\big)^{1/2}}
	\qquad
	\hbox{ for all $n,\,m\in\N$.}
\end{eqnarray*}
Applying this to $s = m + n$ we get
\begin{eqnarray*}
	U_{\theta}(m + n) 
		\geqslant U_\theta(m) + U_\theta(n) - \widehat{K} {\bluee\big(n\ln(n)\big)^{1/2}}
		\qquad
		\hbox{ for all $n,\,m\in\N$.}
\end{eqnarray*}
{\bluee By arguing as in the proof of \Cref{prop:expectation} (see the paragraph after inequality \eqref{inequality sub}), we conclude that $-U_\theta(n) / n$ converges, as $n\to +\infty$, to a limit that we shall call $\overline H(\theta)$. Furthermore, as stressed in \Cref{remark:sub},  there exists a constant $A$ only depending on $|\theta|$,\,$d$,\,$\beta$,\,$\rho$,\,$\delta$ (since $\Lip(\ell)$ and $\left\| f \right\|_\infty$ are bounded from above by $\beta+\delta$) such that 
\[
|U_{\theta}(n)/n+\overline{H}(\theta)| \leq A \ln(n)^{1/2} n^{-1/2}\qquad\FORALL\ n\geq 2.
\]
By  \Cref{teo app Lip estimates} again we have that, for all $(x,\omega)\in\M\times\Omega$, $u_\theta(\cdot,x,\omega)$, and hence $U_\theta$, is $\beta(1+|\theta|)$--Lipschitz in $[0,+\infty)$.  
This proves that $U_\theta$ satisfies the statement \Cref{prop:expectation}. 

The statement of Proposition~\ref{prop Concentration property} remains valid as well. Its proof uses only Propositions~\ref{prop:concentration} and \ref{prop:expectation}, the Lipschitz bounds from Theorem~\ref{teo app Lip estimates}, and the fact that \(H\) is a stationary Hamiltonian in $\Ham$. Consequently, the quantitative estimate in \Cref{Result: Concentration property} follows by the same argument as in Section~\ref{sec: Concentration property}, and Proposition~\ref{prop convergence tool} follows via Lemma~\ref{lemma key implication}. By the reduction arguments in Section~\ref{sec:outline}, the HJ equation \eqref{eq eps hj} homogenizes for such $H$ with effective Hamiltonian $\overline H$. This completes the proof. \qed
}

\appendix

\section{PDE results}\label{sec: app A}

In this appendix we collect the PDE results used in the paper, which are of deterministic nature and which follow from the ones herein stated and proved by regarding at $\omega$ as a fixed parameter. We will denote by $H$ a continuous Hamiltonian defined on $\TM$ and satisfying further assumptions 
that will be specified case by case. Throughout this section, we will denote by 
$\D{LSC}(X)$ and $\D{USC}(X)$ the space of lower semicontinuous and upper semicontinuous real functions on a metric space $X$, respectively. 

\smallskip

Let $T>0$ be fixed and consider the following evolutive HJ equation 
\begin{equation}
\label{appA: eq HJ T}
\tag{HJ}
	\partial_t u + H(x, Du) = 0 \qquad \hbox{in $\Tcyl$.}
\end{equation}

We shall say that a function $v\in\D{USC}(\Tcyl)$ is an (upper semicontinuous) {\em viscosity subsolution} of \eqref{appA: eq HJ T} if, for every  $\phi\in\D{C}^1(\Tcyl)$ such that $v-\phi$ attains a local maximum at $(t_0,x_0)\in \Tcyl$, we have 
\begin{equation}\label{app subsolution test}
\partial_t \phi(t_0,x_0)+H(x_0, D_x \phi(t_0,x_0))\leqslant 0. 
\end{equation}
Any such test function $\phi$ will be called {\em supertangent} to $v$ at $(t_0,x_0)$. 

We shall say that $w\in\D{LSC}(\Tcyl)$ is a (lower semicontinuous) {\em viscosity supersolution} of \eqref{appA: eq HJ T} if, for every $\phi\in\D{C}^1(\Tcyl)$ such that $w-\phi$ attains a local minimum at $(t_0,x_0)\in \Tcyl$, we have 
\begin{equation}\label{app supersolution test}
\partial_t \phi(t_0,x_0)+H(x_0, D_x \phi(t_0,x_0))\geqslant 0. 
\end{equation}
Any such test function $\phi$ will be called {\em subtangent} to $w$ at $(t_0,x_0)$. 
A continuous function on $\Tcyl$ is a {\em viscosity solution} of \eqref{appA: eq HJ T} if it is both a viscosity sub and supersolution. Solutions, subsolutions, and supersolutions will be always intended in the viscosity sense, hence the term {\em viscosity} will be omitted in the sequel.

\subsection{Comparison principles}\label{app: comparison}
%
%
We start by stating and proving a comparison principle, which applies in particular to the case of bounded sub and supersolutions. 
The proof is somewhat standard,  we provide it below for the reader convenience.

\begin{theorem}\label{appA: teo comp}
	Let $H$ be a Hamiltonian satisfying the following assumptions, for some continuity modulus $\omega$:
	\begin{align}
		|H(x, p) - H(x, q)| 
			\leqslant \omega(|p-q|)
			&\qquad \hbox{ for all $x, p, q \in \M$.}\label{appA: condition2 comp}\tag{H2$'$}\\
		|H(x, p) - H(y, p)|
			\leqslant \omega\big(|x - y| (1 + |p|)\big)
			&\qquad \hbox{ for all $x, y, p \in \M$,}\label{appA: condition1 comp}\tag{H4}
	\end{align} 
	Let $v\in\D{USC}([0,T]\times\M)$ and $w\in\D{LSC}([0,T]\times\M)$ be, respectively, a bounded from above subsolution and a bounded from below supersolution of \eqref{appA: eq HJ T}.
	Then, for all $(t,x)\in (0,T)\times \M$,
	\[
	v(t,x)-w(t,x)\leqslant \sup_{\M}\big(v(0,\cdot)-w(0,\cdot)\big) \,.
	\]
\end{theorem}

\begin{proof}
	Since $v$ is bounded from above, up to adding to $v$ a suitable constant, we assume without any loss of generality that 
	$\sup_{\M}\big(v(0,\cdot)-w(0,\cdot)\big) = 0$. 
	The assertion is thus reduced to proving that $v\leqslant w$ in $\Tcyl$. 
	We proceed by contradiction. 
	Assume that $v > w$ at some point of $(0,T)\times \M $. 
	Up to translations, we can assume without loss of generality that this point has the form  $(\overline t, 0)$ for some $\overline t\in (0,T)$. 
	We will construct a point $(x, y, p, q)$ where the continuity assumptions 
	\eqref{appA: condition2 comp}-\eqref{appA: condition1 comp} do not hold, leading to a contradiction. 
	
	For every $\eta, \eps, b > 0$, consider the auxiliary function $\Phi \colon \left(\coTcyl\right)^2 \to \R$ defined by
	\[
		\Phi(t,x,s,y) 
			\coloneqq v(t,x)-w(s,y)-\frac{|x-y|^2}{2\eps} -\frac{(t-s)^2}{2\eps}-\eta\left(\phi(x)+\phi(y)\right)
		-\frac{b}{T-t}-\frac{b}{T-s},
	\]
	where $\phi(x) \coloneqq \sqrt{1+ | x |^2}$.

	Define $\theta \coloneqq v(\overline t, 0) - w(\overline t, 0) > 0$. 
	Then, since $\overline t\in (0,T)$, there exists $\delta > 0$ small enough such that, for all $\eta, b \in(0, \delta )$, 
	\[
		\Phi(\overline t,0,\overline t,0)
			= v(\overline t,0)-w(\overline t,0)-2\eta\phi(0)-\frac{2b}{T-\overline t}> \frac{\theta}{2}\,.
	\]
	Notice that 
	\begin{equation}\label{eq penalization}
		\Phi(t,x,s,y)
			\leqslant 
			M
			-\eta\left(\phi(x)+\phi(y)\right)-\frac{b}{T-t}-\frac{b}{T-s}
			\quad
			\hbox{in $\left(\coTcyl\right)^2$}
	\end{equation}
	with $M \coloneqq (\|v^+\|_{L^\infty(\cTcyl)}+\|w^-\|_{L^\infty(\cTcyl)})$, where we have denoted by $v^+(x) \coloneqq \max\{v(x),0\}$ the positive part of $v$ and by $w^-(x)=\max\{-w(x),0\}$ the negative part of $w$.
	We derive that, for every $\eps > 0$ and $\eta \in (0, \delta )$, there exists $(t_\eps, x_\eps, s_\eps, y_\eps) \in \left(\coTcyl\right)^2$, which also depend on $\eta$, such that 
	\begin{equation}\label{parabolic raddoppio}
		\Phi(t_\eps,x_\eps,s_\eps,y_\eps)=\sup_{(\cTcyl)^2} \Phi \geqslant \Phi(\overline t,0,\overline t,0)> \frac{\theta}{2}.
	\end{equation}
	In view of \eqref{eq penalization} we infer 	
	\begin{equation}\label{eq bounded p}
		\eta\left(\phi(x_\eps)+\phi(y_\eps)\right)
			+\frac{b}{T-t_\eps}+\frac{b}{T-s_\eps}
			\leqslant 
			\tilde M 
		\qquad
		\hbox{and}
		\qquad
		\dfrac{|x_\eps-y_\eps|}{\eps}
			+ \dfrac{|t_\eps-s_\eps|}{\eps}
			\leqslant \sqrt{\dfrac{2\tilde M}{\eps}}\,
	\end{equation}
	with $\tilde M \coloneqq M- \theta / 2$. 
	From the first inequality in \eqref{eq bounded p} we derive that exist constants $T_{b}\in (0,T)$, depending on $b \in (0, \delta )$, and $\rho_\eta>0$, depending on $\eta$, such that 
	$t_\eps,s_\eps\in [0, T_{b}]$ and $x_\eps,y_\eps \in B_{\rho_\eta}$. 
	For every fixed $\eta\in (0, \delta )$,  from  \cite[Lemma 3.1]{users} we derive that, up to subsequences,  
	\begin{equation}\label{parabolic limits approximating points}
		\lim_{\eps\to 0} (t_\eps,x_\eps,s_\eps,y_\eps)=(t_0,x_0,t_0,x_0)
		\qquad\hbox{and}\qquad
		\lim_{\eps\to 0}\frac{|x_\eps-y_\eps|^2}{\eps} = 0
	\end{equation}
	for some $(t_0,x_0)\in\cTcyl$ satisfying 
	\begin{equation}\label{parabolic maximum point}
		v(t_0,x_0)-w(t_0,x_0)-2\eta\phi(x_0)-\frac{2b}{T-t_0}
			= \sup_{(t,x)\in\Tcyl} \Phi(t,x,t,x)
			> \frac{\theta}{2}.
	\end{equation}
	In particular, such a point  $(t_0,x_0)$ actually lies in $\Tcyl$, i.e., $t_0>0$, since \eqref{parabolic maximum point} implies $v(t_0,x_0)-w(t_0,x_0)>0$. 
	For every fixed $\eta\in(0,\delta)$, choose $\eps_\eta>0$ small enough so that $(t_\eps,x_\eps)$ and $(s_\eps,y_\eps)$ both belong to $\Tcyl$ when $\eps\in (0,\eps_\eta)$.  The function 
	\[
	\varphi(t,x) \coloneqq w(s_\eps,y_\eps)+\frac{|x-y_\eps|^2}{2\eps} +\eta\left(\phi(x)+\phi(y_\eps)\right)+\frac{|t-s_\eps|^2}{2\eps}+\frac{b}{T-t}
	\]
	is a supertangent to $v(t,x)$ at the point $(t_\eps,x_\eps)$ and $v$ is a subsolution of \eqref{appA: eq HJ T}, while the function 
	\[
	\psi(s,y) \coloneqq v(t_\eps,x_\eps)-\frac{|x_\eps-y|^2}{2\eps}  -\eta\left(\phi(x_\eps)+\phi(y)\right)-\frac{|t_\eps-s|^2}{2\eps}-\frac{b}{T-s}
	\]
	 is a subtangent to $w(s,y)$ at the point $(s_\eps,y_\eps)$ and $w$ is a subsolution of \eqref{appA: eq HJ T}. We infer  
	\begin{eqnarray}\label{comp parabolic subsol2}
	\dfrac{t_\eps-s_\eps}{\eps}
	+
	H\left(x_\eps,    p_\eps+\eta D\phi(x_\eps)\right)&\leqslant& -c,\label{comp parabolic subsol2}\\
	\dfrac{t_\eps-s_\eps}{\eps}
	+
	H\left(y_\eps,    p_\eps-\eta D\phi(y_\eps)\right)&\geqslant& c,\label{comp parabolic subsol1}
	\end{eqnarray}
	where we have set
	\[
		c \coloneqq \frac{b}{T^2}
		\qquad \hbox{and}\qquad
		p_\eps \coloneqq \frac{x_\eps-y_\eps}{\eps}.
	\]
	By subtracting \eqref{comp parabolic subsol2} from \eqref{comp parabolic subsol1}, we get, according to assumptions 
	\eqref{appA: condition2 comp}-\eqref{appA: condition1 comp}, 
	\begin{align}\label{app estimate 3}
		0 
			< 2c
			\leqslant H\left(y_\eps,     p_\eps-\eta D\phi(y_\eps)\right)-H\left(x_\eps,     p_\eps+\eta D\phi(x_\eps)\right)\nonumber
			\leqslant \omega\big(|x_\eps-y_\eps|(1+   |p_\eps|)\big) + 2 \omega\left(\eta\right).
	\end{align}
	Sending $\eps\to 0^+$, we have that $\omega\big(|x_\eps-y_\eps|(1+   |p_\eps|)\big)$ vanishes in view of \eqref{parabolic limits approximating points}.
	Sending $\eta\to 0^+$, we have that $\omega\left(\eta\right)$ vanishes.
	But this is a contradiction since $0 < c$.
\end{proof}

Next, we provide a result which is usually used to prove the finite speed of propagation of first order HJ equations, see for instance \cite[Lemma 5.3]{barles2}. 

\begin{proposition}\label{appA: finite speed}
	Let  $H$ be a Hamiltonian satisfying \eqref{appA: condition1 comp} and condition {\rm (H2)},  for some constant $\beta>0$.
	Let $v\in\D{USC}([0,T]\times\M)$ and $w\in\D{LSC}([0,T]\times\M)$ be, respectively, a sub and a supersolution of \eqref{appA: eq HJ T}. Then the function $u \coloneqq v-w$ is an upper semicontinuous subsolution to 
	\begin{equation}\label{appA: geometric eq}
		\partial_t u -\beta|Du| = 0 \qquad \hbox{in $\Tcyl$.}
	\end{equation}
\end{proposition}

\begin{proof}
	Let $\varphi\in C^1(\Tcyl)$ be supertangent to $u$ in a point $(t_0,x_0)\in\Tcyl$ and let us assume that $(t_0,x_0)\in \Tcyl$ is a strict local maximum point of $u-\varphi$.  Let us choose $r>0$ such that the open ball $B \coloneqq B_{r}\big((t_0,x_0)\big)$ of radius $r>0$ centred at $(t_0,x_0)$ is compactly contained in $\Tcyl$ and $(t_0,x_0)$ is the only maximum point of $u-\varphi$ in $\overline B$. Let us introduce the function 
	\[
		\Phi(t,x,s,y) \coloneqq v(t,x)-w(s,y)-\frac{|x-y|^2}{2\eps}-\frac{|t-s|^2}{2\eps} - \varphi(t,x)
			\qquad\hbox{ for $(t, x, s, y) \in \overline B \times \overline B$.}
	\]
	By uppersemicontinuity of $\Phi$ and compactness of the domain, the maximum of $\Phi$ on $\overline B\times\overline B$ is attained at (at least) a point $(t_\eps,x_\eps,s_\eps,y_\eps)\in \overline B\times\overline B$. In view of \cite[Lemma 3.1]{users}, we infer that 
	\begin{equation}\label{appA: limit points}
		\lim_{\eps\to 0} (t_\eps,x_\eps,s_\eps,y_\eps)=(t_0,x_0,t_0,x_0)
		\qquad\hbox{and}\qquad
		\lim_{\eps\to 0}\frac{|x_\eps-y_\eps|^2}{\eps}=0
	\end{equation}
	Choose $\eps_0>0$ small enough so that $(x_\eps,t_\eps)$, $(y_\eps,s_\eps)$ belong to $B$ for every $\eps\in (0,\eps_0)$. The function 
	\[
	\psi_1(t,x) \coloneqq w(s_\eps,y_\eps)+\frac{|x-y_\eps|^2}{2\eps}+\frac{|t-s_\eps|^2}{2\eps}+\varphi(t,x) 
	\]
	is a supertangent to $v(t,x)$ at $(t_\eps,x_\eps)$ and $v$ is a subsolution to \eqref{appA: eq HJ T}, hence 
	\begin{equation}\label{appA: eq speed1}
	\dfrac{t_\eps-s_\eps}{\eps}+\partial_t \varphi  (t_\eps,x_\eps)+H\left(x_\eps,p_\eps+D_x \varphi (t_\eps,x_\eps)\right)\leqslant 0,
	\end{equation}
	where we have set $\displaystyle p_\eps \coloneqq \frac{x_\eps-y_\eps}{\eps}$. Analogously, 
	\[
		\psi_2(s,y) \coloneqq v(t_\eps,x_\eps)-\frac{|x_\eps-y|^2}{2\eps}-\frac{|t_\eps-s|^2}{2\eps}-\varphi(t_\eps,x_\eps)
	\]
	is a subtangent to $w(s,y)$ at the point $(s_\eps,y_\eps)$  and $w$ is a supersolution to \eqref{appA: eq HJ T}, hence 
	\begin{equation}\label{appA: eq speed2}
		\dfrac{t_\eps-s_\eps}{\eps}+H\left(y_\eps,p_\eps\right)\geqslant 0,
	\end{equation}
	By subtracting \eqref{appA: eq speed2} from \eqref{appA: eq speed1} and by taking into account conditions {\rm (H2)} and \eqref{appA: condition1 comp}, we get 
	\begin{eqnarray*}
		0 
			&\geqslant& 
			\partial_t \varphi  (t_\eps,x_\eps)+H\left(x_\eps,p_\eps+D_x \varphi (t_\eps,x_\eps)\right)
			-
			H\left(y_\eps,p_\eps\right)\\
			&\geqslant& 
			\partial_t \varphi  (t_\eps,x_\eps)-\beta|D_x \varphi (t_\eps,x_\eps)|+H\left(x_\eps,p_\eps\right)
			-
			H\left(y_\eps,p_\eps\right)\\
			&\geqslant& 
			\partial_t \varphi  (t_\eps,x_\eps)-\beta|D_x \varphi (t_\eps,x_\eps)|
			-\omega\big(|x_\eps-y_\eps|(1+|p_\eps|)\big).
	\end{eqnarray*}
	Now we send $\eps\to 0^+$ to get, in view of \eqref{appA: limit points}, 
	\[
		0\geqslant \partial_t \varphi  (t_0,x_0)-\beta|D_x \varphi (t_0,x_0)|,
	\]
	as it was to be shown.
\end{proof}

With the aid of the previous proposition, we can prove the following version of the Comparison Principle for unbounded sub and supersolutions.

\begin{theorem}\label{appA: teo comp2}
	Let $H$ be a Hamiltonian satisfying assumptions {\rm (H2)} and \eqref{appA: condition1 comp}. 
	Let $v \in \D{USC}([0, T] \times \M)$ and $w \in \D{LSC}([0, T] \times \M)$ be, respectively, a sub and a supersolution of \eqref{appA: eq HJ T}.
	Then,
	\[
		v(t,x)-w(t,x)\leqslant \sup_{\M}\big(v(0,\cdot)-w(0,\cdot)\big)
			\qquad\hbox{ for every  $(t,x)\in (0,T)\times \M$.} 
	\]
\end{theorem}
\begin{proof}
We can assume $\sup_{\M}\big(v(0,\cdot)-w(0,\cdot)\big)<+\infty$, otherwise the assertion is trivially satisfied. Up to adding an appropriate constant to $v$, we can furthermore assume, without loss of generality, that 
$\sup_{\M}\big(v(0,\cdot)-w(0,\cdot)\big)=0$. In view of Proposition~\ref{appA: finite speed}, the function $u \coloneqq v-w$ is an upper semicontinuous subsolution of \eqref{appA: geometric eq} satisfying $u(0,\cdot)\leqslant 0$ in $\M$. Let $\Psi \colon \R\to\R$ be a $C^1$, strictly increasing and bounded function satisfying $\Psi(0)=0$. It is easily seen, due to the positive 1--homogeneity of equation \eqref{appA: geometric eq}, that the function {\bluee $v \coloneqq (\Psi\circ u)(t,x)$ is a bounded, upper semicontinuous subsolution to \eqref{appA: geometric eq} satisfying $v(0,\cdot)\leqslant 0$ on $\M$.} By applying Theorem~\ref{appA: teo comp} 
{\bluee with $w(t,x)\equiv 0$}, 
we derive that $\bluee v=(\Psi\circ u)\leqslant w=0=\Psi(0)$ in $\cTcyl$. The assertion follows by the strict monotonicity of $\Psi$. 
\end{proof}

\subsection{Existence results and Lipschitz estimates for solutions}\label{app:existence results} 
Throughout this subsection, we will assume that the (deterministic) Hamiltonian $H$ satisfying the following assumptions for some constants $\beta_1,\beta_2,\beta_3>0$:\smallskip
\begin{itemize}
\item[(H1$^*$)] $|H(x,p )|\leqslant \beta_1\left(1+|p|\right) \qquad \qquad \qquad \qquad\  \text{for all}\ (x, p) \in \TM$;\smallskip
\item[(H2$^*$)] $|H(x,p ) - H(x,q )| \leqslant \beta_2 |p - q| \qquad\qquad \text{for all}\ x, p, q\in\M$;\smallskip
\item[(H3$^*$)] $|H(x,p ) - H(y,p )| \leqslant \beta_3 |x - y| \qquad \quad \quad\text{for all}\ x, y, p\in\M$.
\end{itemize}
\smallskip

We begin with the following existence and uniqueness result, where the uniqueness part is guaranteed by Theorem~\ref{appA: teo comp2}.

\begin{theorem}
\label{teo appendix existence solution}
	Let $g \in\UC(\M)$. For every $T>0$,  the following problem 
	\begin{align*}\tag{HJP}
	\label{eq app HJP}
	\left\{ \begin{aligned}
			&\partial_t u + H(x, Du) = 0 &\hbox{in $\Tcyl$} \\
			&u(0, \cdot) = g(\cdot) &\hbox{on $\RR^d$}
		\end{aligned}
	\right.
	\end{align*}
has a unique viscosity solution in $\CC(\cTcyl)$. Furthermore, this solution belongs to $\UC(\cTcyl)$. 
\end{theorem}

\begin{proof}
The case $g\in\BUC(\M)$ is proved in  \cite[Theorem 7.1]{barles2}. 
Let us then assume $g\in\UC(\M)$. Pick 
$\psi\in\CC^{1,1}(\M)\cap \Lip(\M)$  such that $\|\psi-g\|_{\infty}\leqslant 1$. In view of the previous step, the Cauchy problem   \eqref{eq app HJP} with $\tilde H(x,p) \coloneqq H(x,D\psi(x)+p)$ and $g-\psi$ in place of $H$ and $g$, respectively, admits a unique solution $\tilde u(t,x)$ in $\CC(\cTcyl)$. Furthermore,  
$\tilde u(t,x)$ in $\UC(\cTcyl)$.
We derive that $u(t,x) \coloneqq \tilde u(t,x)+\psi(x)$ belongs to $\UC(\cTcyl)$ and is a solution of the original Cauchly problem \eqref{eq app HJP}.  
\end{proof}

We proceed to show suitable Lipschitz bounds for the solution of \eqref{eq app HJP} when the initial datum is Lipschitz. 

\begin{theorem}\label{teo app Lip estimates}
Let $g\in\Lip(\M)$ and let $u$ be the unique continuous function in $\ccyl$ which solves the Cauchy problem \eqref{eq app HJP} for every $T>0$. 
Then $u$ is Lipchitz in $\cTcyl$ for every $T>0$. More precisely 
\[
\| D_x u\|_{L^\infty(\ccTcyl)}\leqslant \left(\beta_3 {\bluee T}+\Lip(g)\right),
\qquad
\|\partial_t u\|_{L^\infty(\ccTcyl)}\leqslant \beta_1 \left(1+\|Dg\|_\infty\right).
\]
\end{theorem}

\begin{proof}
(i) Let us fix $\hh\in\M$ and set 
\[
v_\hh(t,x) \coloneqq u(t,x+\hh)-\beta_3|\hh| t,
\quad
w_\hh(t,x) \coloneqq u(t,x+\hh)+\beta_3|\hh| t
\qquad
\hbox{for every $(t,x)\in\ccyl$.}
\]
By exploiting assumption (H2$^*$), it is easily seen that $v_h$ and $w_h$ are, respectively, a viscosity sub and supersolution to \eqref{appA: eq HJ T}. 
Indeed, the following inequalities hold in the viscosity sense:
\begin{eqnarray*}
\partial_t v_h+H(x,D v_h)
&=&
\partial_t u(t,x+h)-\beta_3|h|+H(x,D u(t,x+h))\\
&\leqslant& 
\partial_t u(t,x+h)+H(x+h,D u(t,x+h))
\leqslant 
0
\qquad
\hbox{in $\Tcyl$},
\end{eqnarray*}
thus showing that $v_h$ is a subsolution of \eqref{appA: eq HJ T}. The assertion for $w_h$ can be proved analogously.  
By the Comparison Principle, see Theorem~\ref{appA: teo comp2}, we infer that, for all $(t,x)\in\cTcyl$, 
\begin{eqnarray*}
u(t,x)-v_h(t,x) &\leqslant& u(0,x)-v_h(0,x)=g(x)-g(x+h)\leqslant \Lip(g)|h|\\
w_h(t,x)-u(t,x) &\leqslant& w_h(0,x)-u(0,x)=g(x+h)-g(x)\leqslant \Lip(g)|h|,
\end{eqnarray*}
namely 
\begin{equation}\label{eq app x-Lip bound}
|u(t,x+h)-u(t,x)|\leqslant \left(\beta_3 t+\Lip(g)\right) |h|\qquad \hbox{for all $(t,x)\in\ccyl$,}
\end{equation}
thus showing the first assertion.\smallskip

(ii) Let us first assume that $g\in\Lip(\M)\cap\CC^1(\M)$. By assumption (H1$^*$) we have 
\begin{equation}\label{eq app pre t-Lip bound}
|H(x,Dg(x))| 
\leqslant
\beta_1 \left(1+ |Dg(x)|\right) 
\leqslant 
\beta_1 \left(1+ \|Dg\|_\infty\right) 
\qquad
\hbox{for all $x\in\M$.}
\end{equation}
For notational simplicity, let us denote by $M$ the most right-hand side term in the above inequality. It is easily seen that the functions 
\[
\underline u(t,x) \coloneqq g(x)-Mt,
\quad
\overline u(t,x) \coloneqq g(x)+Mt,
\qquad
(t,x)\in\ccyl,
\]
are, respectively, a sub and a supersolution of  \eqref{eq app HJP} for every $T>0$. By the Comparison Principle, see Theorem~\ref{appA: teo comp2}, we infer that 
$
\underline u(t,x)
\leqslant 
u(t,x)
\leqslant 
\overline u(t,x)
\ 
\hbox{for all $(t,x)\in\ccyl$},
$
i.e., 
\[
\|u(t,\cdot)-g\|_\infty \leqslant Mt\qquad\hbox{for all $t\geqslant 0$.}
\]
By applying the Comparison Principle again we get 
\begin{equation}\label{eq app t-Lip estimate}
\|u(t+h,\cdot)-u(t,\cdot)\|_\infty 
\leqslant 
\|u(h,\cdot)-u(0,\cdot)\|_\infty 
\leqslant 
Mh
= 
\beta_1 \left(1+ \|Dg\|_\infty\right)\,h 
\quad\hbox{for all $t,h\geqslant 0$,}
\end{equation}
meaning that $u$ is $\beta_1 \left(1+ \|Dg\|_\infty\right)$--Lipschitz in $t$.

The case $g\in\Lip(\M)$ can be obtained by approximation. Let us denote by $g_n$ the mollification of $g$ via a standard convolution kernel and by $u_n$ the solution to the Cauchy problem \eqref{eq app HJP} with $g_n$ in place of $g$. Since $\|Dg_n\|_\infty\leqslant \|Dg\|_\infty$ for every $n\in\N$, we deduce from what proved above that the functions $u_n$ are $\beta_1 \left(1+ \|Dg\|_\infty\right)$--Lipschitz in $t$ and $\left(\beta_3 t+\Lip(g)\right)$--Lipschitz in $x$, for every $n\in\N$. By the Ascoli--Arzel\`a Theorem, {\bluee the stability of the notion of viscosity solution and the uniqueness of the continuous solution to the Cauchy problem associated with \eqref{eq app HJP}}, we infer that the functions $(u_n)_n$ are converging, locally uniformly in $\ccyl$, to the solution $u$ of  \eqref{eq app HJP} with initial datum $g$. We derive that $u$ satisfies \eqref{eq app t-Lip estimate} as well, as it was to be shown. 
\end{proof}

\subsection{Differential Games}\label{app:PDE} Throughout this subsection, we will  work with a deterministic Hamiltonian $H$ of the form 
\[
	H(x,p) \coloneqq \max_{b \in B} \min_{a \in A} \left\{ -\ell(x, a, b) - \langle f(a,b), p \rangle \right\} \qquad \text{for all } (x,p) \in \TM,
\]
where $A, B$ are compact subsets of $\R^m$, for some integer $m$, $f \colon A\times B \to \R^d$ is a continuous {\bluee vector--valued function}, and the  running cost $\ell \colon \R^d \times A \times B \to \R$ is a bounded and continuous function satisfying assumption ($\ell_2$) appearing in Section~\ref{sec:assumptions}. We shall denote by $\| \ell \|_\infty$ the $L^\infty$--norm of $\ell$ on  $\R^d \times A \times B$. 
The Hamiltonian $H$ clearly satisfies the properties (H1)--(H3) listed in Section~\ref{sec:assumptions}. 

For such a Hamiltonian, the solutions to the Cauchy problem \eqref{eq app HJP} can be represented trough a max-min formula provided by Differential Game Theory, see~\cite{EvSo84}. To this aim, we first observe that $u(t,x)$ is a viscosity solution of \eqref{eq app HJP} if and only $\check u (t,x) \coloneqq -u(T-t,x)$ is a viscosity solution of equation \eqref{appA: eq HJ T} with $\check H(x,p) \coloneqq H(x,-p)$ in place of $H$ in the sense adopted in~\cite{EvSo84}, see items (a)-(b) at the end of page 774, and satisfying the terminal condition $\check u(T,x)=-g(x)$, cf. \cite[problem (HJ)]{EvSo84}. 
Let us denote by 
\begin{align*}
	\AA(T) \coloneqq \left\{ a \colon [0,T] \to A \, : \, a \text{ measurable} \right\},\qquad
	\BB(T) \coloneqq \left\{ b \colon [0,T] \to B \, : \, b \text{ measurable} \right\}.
\end{align*}
The sets $A$ and $B$ are to be regarded as action sets for Player~1 and~2, respectively. 
A {\em nonanticipating strategy} for Player~1 is a function $\alpha \colon \BB(T) \to \AA(T)$ such that, 
for all $b_1, b_2 \in \BB(T)$ and $\tau \in [0, T]$, 
\[
	b_1(\cdot) = b_2(\cdot) \quad \text{in } [0,\tau]
		\quad
		\Rightarrow
		\quad
	\alpha[b_1] (\cdot) = \alpha[b_2](\cdot) \quad \text{in } [0,\tau] \,.
\]
We will denote by $\Gamma(T)$ the family of such nonanticipating strategies for Player~1. 
For every $(t,x)\in\cyl$, the value function associated with this Differential Game is defined as  
\begin{equation}
\label{def app value function}
	v(t,x) \coloneqq \sup_{\alpha \in \Gamma(t)}\inf_{b \in \BB(t)} \left\{
		\int_0^t \ell(y_x(s),\alpha[b](s),b(s))\,ds + g(y_x(t)) 
	\right\},
\end{equation}
where $y_x \colon [0,t] \to \M$ is the solution of the ODE 
\begin{align}\tag{ODE}\label{eq ODE}
	\begin{cases}
		\dot y_x(s) = f(\alpha[b](s),b(s))\qquad\hbox{in $[0,t]$} \\
		y_x(0) = x.
	\end{cases}
\end{align} 

\begin{prop}\label{prop properties value function}
\ 
\begin{itemize}
\item[{\em (i)}] Let $g\in W^{1,\infty}(\M)$.   Then $v\in W^{1,\infty}(\cTcyl)$ for every fixed $T>0$. 
More precisely, for every $x, \hat{x}\in\M$ and $t, \hat{t} \in (0, T)$, we have 
\begin{align*}
		&|v(t,x)|
			\leqslant T  \| \ell \|_\infty + \|g\|_\infty \,,\\
		&\bluee |v(t, x) - v(\hat{t}, \hat{x})|
			\leqslant \left( T\Lip(\ell)+ \Lip(g) \right) |x - \hat{x}| 
				+ 	
				\Big( \| \ell \|_\infty + \| f \|_\infty \big(T\,\Lip(\ell)+\Lip(g) \big)\Big) |t - \hat{t}| \,.
	\end{align*}
\item[{\em (ii)}] Let $g\in \BUC(\M)$. Then $v\in \BUC(\cTcyl)$ for every fixed $T>0$.
\end{itemize}
\end{prop}

\begin{proof}
{\bluee 
The first part of item (i) follows directly from \cite[Theorem 3.2]{EvSo84} after observing that $v(t,x)=-V(T-t,x)$, where $V$ is the function given by (2.6) in  \cite{EvSo84} with $Z \coloneqq A$,  $Y \coloneqq B$, and  $-\ell$ and $-g$ in place of $\ell$ and $g$, respectively. The first inequality in (i) can be easily deduced from \eqref{def app value function} in view of the uniform bounds on $\ell$ and $g$. To derive the second inequality, we have to prove the same kind of Lipschitz bounds for the function $V$. This can be achieved by arguing as in \cite[Proof of Theorem 3.2]{EvSo84}. We sketch the proof here and refer to \cite{EvSo84} for the details. 

Let us consider the lower value 
\[
V(t,x) \coloneqq 
\inf_{\alpha}\sup_{b} \left\{
		\int_t^T -\ell(y_x(s),\alpha[b](s),b(s))\,ds - g(y_x(T)) 
	\right\},
\]
for the dynamics 
\[
\dot y_x(s)=f\!\big(\alpha[b](s),b(s)\big)\quad\hbox{in $[t,T]$},\qquad y_x(t)=x,
\]
with \(b\) and \(\alpha[b]\) taken from the corresponding admissible control classes of Players~2 and~1.\smallskip 

\paragraph{\underline{Lipschitz continuity in $x$ (time $t$ fixed)}}
Fix $t\in (0,T)$ and $x_1,x_2\in\M$. Run the {same} controls $(\alpha,\beta)$ after time $t$ for both initial conditions.
Since $f$ does not depend on $x$, we have 
\[
|y_1(s)-y_2(s)|=|x_1-x_2|\qquad\FORALL s\in[t,T].
\]
This replaces the Grönwall inequality used between (3.16)--(3.17) in \cite{EvSo84}.
Proceeding as in (3.17)--(3.20) in \cite{EvSo84}, the running and terminal parts satisfy
\begin{flalign*}
\bigg|\int_t^T \!\ell\big(\alpha(s),\beta(s),y_1(s)\big)-\ell\big(\alpha(s),\beta(s),y_2(s)\big)\,ds\bigg|
&\le \Lip(\ell)\,(T-t)\,|x_1-x_2|,\\
|g(y_1(T))-g(y_2(T))|
&\le \Lip(g)\,|x_1-x_2|.
\end{flalign*}
Taking the $\inf_\alpha\,\sup_\beta$ over admissible controls yields
\[
|V(t,x_1)-V(t,x_2)|\le \Big(\Lip(\ell)\,(T-t)+\Lip(g)\Big)\,|x_1-x_2|,
\]
and hence $\sup_{t\in[0,T]}\,\Lip_x V(t,\cdot)\le T\,\Lip(\ell)+\Lip(g)$.\smallskip

\paragraph{\underline{Lipschitz continuity in $t$ (state $x$ fixed)}}
Fix $t_1<t_2$ in $(0,T)$. As in (3.17) of \cite{EvSo84}, split the payoff difference into:
\begin{itemize}
  \item \emph{Short interval $[t_1,t_2]$:} 
  \(
 \displaystyle \Big|\int_{t_1}^{t_2}\ell\big(\alpha(s),\beta(s),y_{t_1}(s)\big)\,ds\Big|\le \|\ell\|_\infty\,|t_2-t_1|.
  \)\smallskip
  \item \emph{Overlap $[t_2,T]$:} run the same controls after $t_2$. At time $t_2$,
  \[
  y_{t_1}(t_2)-x=\int_{t_1}^{t_2} f(\alpha(s),\beta(s))\,ds,
  \qquad
  \big|y_{t_1}(t_2)-x\big|\le \|f\|_\infty\,|t_2-t_1|.
  \]
  Since $f$ is $x$--independent, this offset is preserved for all $s\ge t_2$, i.e.,
  \[
  |y_{t_1}(s)-y_{t_2}(s)|=\big|y_{t_1}(t_2)-x\big|\le \|f\|_\infty\,|t_2-t_1|,
  \qquad s\in[t_2,T].
  \]
  This replaces the Grönwall growth used between (3.18)–(3.21) in \cite{EvSo84}.\smallskip 
\end{itemize}
  Consequently, exactly as in (3.18)--(3.21) in  \cite{EvSo84},
  \[
  \begin{aligned}
  \bigg|\int_{t_2}^{T}\!\ell\big(s,y_{t_1}(s)\big)-\ell\big(s,y_{t_2}(s)\big)\,ds\bigg|
  &\le \Lip(\ell)\,(T-t_2)\,\|f\|_\infty\,|t_2-t_1|,\\
  |g(y_{t_1}(T))-g(\dot y{t_2}(T))|
  &\le \Lip(g)\,\|f\|_\infty\,|t_2-t_1|.
  \end{aligned}
  \]
Collecting the pieces and passing to the $\inf_\alpha\,\sup_\beta$ over admissible controls gives
\[
|V(t_1,x)-V(t_2,x)|
\le \Big(\|\ell\|_\infty+\|f\|_\infty\big(\Lip(\ell)\,(T-t_2)+\Lip(g)\big)\Big)\;|t_2-t_1|.
\]
In particular,
\[
\sup_{(t,x)\in[0,T]\times\M}\Lip_t V
\le
\|\ell\|_\infty+\|f\|_\infty\big(T\,\Lip(\ell)+\Lip(g)\big).
\]
This yields the Lipschitz bounds for $v$ in $(t,x)$ stated above.\smallskip
}

Let us prove (ii). 
Let $(g_n)_n$ be a sequence of functions in $W^{1,\infty}(\M)$ such that $\|g_n-g\|_{L^\infty(\M)}\to 0$ as $n\to +\infty$.  
For every fixed $(t,x)\in \Tcyl$, any fixed strategy $\alpha\in\Gamma(t)$ and every control $b\in\BB(t)$, let us set 
\[
J[\alpha,b,g](t,x) \coloneqq \int_0^t \ell(y_x(s),\alpha[b](s),b(s))\,ds+ g(y_x(t)),
\] 
where $y_x \colon [0,t]\to\M$ is the solution of \eqref{eq ODE}. We have 
\[
J[\alpha,b,g_n](t,x)-\|g_n-g\|_{\infty}
\leqslant 
J[\alpha,b,g](t,x)
\leqslant
J[\alpha,b,g_n](t,x)+\|g_n-g\|_{\infty}.
\]
By taking the infimum with respect to $b\in\BB(t)$ and then the supremum with respect to $\alpha\in\Gamma(t)$, we infer 
\begin{eqnarray*}
\sup_{\alpha\in\Gamma(t)}\inf_{b\in\BB(t)} J[\alpha,\hat b,g_n](t,x)-\|g_n-g\|_{\infty}
&\leqslant& 
\sup_{\alpha\in\Gamma(t)}\inf_{b\in\BB(t)}  J[\alpha,\hat b,g](t,x)\\
&\leqslant&
\sup_{\alpha\in\Gamma(t)}\inf_{b\in\BB(t)}
J[\alpha,b,g_n](t,x)+\|g_n-g\|_{\infty}.
\end{eqnarray*}
This means that $|v_n(t,x)-v(t,x)|\leqslant \|g_n-g\|_{\infty}$ for all $(t,x)\in\cTcyl$, where $v_n$ and $v$ denote the value function associated to $g_n$ and $g$, respectively.  As a uniform limit of a sequence of equi--bounded Lipschitz functions, we conclude that  $v$ belongs to $\BUC(\cTcyl)$. 
\end{proof}

Via the same {\bluee argument} used in the proof of Proposition~\ref{prop properties value function}-(i), we derive from \cite[Theorem 3.1]{EvSo84} the following fact, known as Dynamic Programming Principle.

\begin{theorem}[Dynamic Programming Principle]
\label{teo dynamic programming principle}
{\bluee Let $g \in\BUC(\M)$. For every fixed $x \in \M$ and $0 < \tau < T$, we have}  
	\begin{equation} \label{eq:DP}
		v(T, x) = \sup_{\alpha\in\Gamma(T - \tau)} \inf_{b\in\BB(T - \tau)} \left\{
			\int_0^{T - \tau} \ell(y_x(s), \alpha[b](s), b(s)) \,ds + v(\tau, y_x(T - \tau)) 
		\right\} \,.
	\end{equation}
\end{theorem}

The following holds. 
\begin{theorem}
\label{teo appendix representation}
Let $g \in\BUC(\M)$. For every fixed $T>0$, the unique continuous solution  of \eqref{eq app HJP} is given by 
\eqref{def app value function}. 	
\end{theorem}

\begin{proof} 
When $g$ is additionally assumed Lipschitz continuous, the assertion follows 
directly from \cite[Theorem 4.1]{EvSo84} via the same change of variables used in the proof of Proposition~\ref{prop properties value function}-(i) and in view of what remarked at the beginning of this subsection. 
Let us assume $g\in \BUC(\M)$ and pick a sequence of functions $g_n\in W^{1,\infty}(\M)$ such that $\|g_n-g\|_{L^\infty(\M)}\to 0$ as $n\to +\infty$. Let us denote by $v$ and $v_n$ the value functions associated via \eqref{def app value function} to $g$ and $g_n$, respectively. Arguing as in the proof of 
Proposition~\ref{prop properties value function}-(ii), we derive  
$|v_n(t,x)-v(t,x)|\leqslant \|g_n-g\|_{\infty}$ for all $(t,x)\in\cTcyl$. From the previous step we know that $v_n$ solves 
\eqref{eq app HJP} with initial datum $g_n$. By the stability of the notion of viscosity solution, we conclude that $v$ solves \eqref{eq app HJP}. 
\end{proof}
4
We now extend the previous result to the case of initial data that are not necessarily bounded. The result is the following. 

\begin{theorem}\label{teo HJ equation extended}
Let $g \in \UC(\M)$. For every fixed $T>0$, the unique viscosity solution $u\in \CC(\cTcyl)$  of the Cauchy problem \eqref{eq app HJP} is given by the representation formula \eqref{def app value function}. {\bluee Furthermore, $u$ satisfies the Dynamic Programming Principle \eqref{eq:DP}.}
\end{theorem}

\begin{proof}
Let us pick $\psi\in\CC^{1,1}(\M)\cap \Lip(\M)$  such that $\|\psi-g\|_{\infty}\leqslant 1$. In view of Theorem~\ref{teo appendix representation}, the unique solution $\tilde u(t,x)$ in $\CC(\cTcyl)$ of the Cauchy problem   \eqref{eq app HJP} with $\tilde H(x,p) \coloneqq H(x,D\psi(x)+p)$ and $g-\psi$ in place of $H$ and $g$, respectively,  is given by the formula \eqref{def app value function} with 
$\ell(y_x(s),\alpha[b](s),b(s))+\langle f(\alpha[b](s),b(s)),D\psi(y_x(s))\rangle$ and $(g-\psi)(y_x(t))$ in place of 
{\bluee $\ell(y_x(s),\alpha[b](s),b(s))$ and $g(y_x(t))$, respectively.} 
For every fixed strategy $\alpha\in\Gamma(t)$ and every control $b\in\BB(t)$, we have 
\begin{eqnarray*}
&&\int_0^t \Big( \ell(y_x(s),\alpha[b](s),b(s)) +  \langle f(\alpha[b](s),b(s)),D\psi(y_x(s))\rangle\Big)\,ds
=\\
&&\int_0^t   \left( \ell(y_x(s),\alpha[b](s),b(s)) +  \dfrac{d}{d s} \psi(y_x(s)) 
\right) \,ds
=
\int_0^t   \ell(y_x(s),\alpha[b](s),b(s)) \,ds
+
\psi(y_x(t))-\psi(x).
\end{eqnarray*}
We infer 
\begin{eqnarray*}
\tilde u(t,x)&=&\sup_{\alpha \in \Gamma(t)}\inf_{b \in \BB(t)} \left\{
		\int_0^t \Big( \ell(y_x(s),\alpha[b](s),b(s)) +  \langle f(\alpha[b](s),b(s)),D\psi(y_x(s))\rangle\Big)\,ds  + (g-\psi)(y_x(t)) \right\}\\
		&=& \sup_{\alpha \in \Gamma(t)}\inf_{b \in \BB(t)} \left\{
		\int_0^t \ell(y_x(s),\alpha[b](s),b(s)) \,ds  + g(y_x(t)) \right\}-\psi(x).
\end{eqnarray*}
The first assertion readily follows after observing that the function $u(t,x) \coloneqq \tilde u(t,x)+\psi(x)$ is the continuous solution of the original Cauchly problem \eqref{eq app HJP}.  {\bluee The second assertion can be derived via the same argument from the fact that $\tilde u$ satisfies the Dynamic Programming Principle \eqref{eq:DP}.}
\end{proof}

\section{Proof of Theorem~\ref{teo reduction}}\label{Section: Reduction}

In this section we prove Theorem~\ref{teo reduction}. The result follows from a couple of preliminary propositions of deterministic nature, with $\omega$ treated as a fixed parameter. We will therefore omit it from our notation. 
We start with the following result.

\begin{prop}\label{app prop dense}
Let $H:\TM\to\R$ be a continuous Hamiltonian satisfying conditions (H1)--(H3) for some $\beta>0$.  For every fixed $\theta \in \M$ and $\eps > 0$, let us denote by $\tilde u^\eps_\theta$ the unique continuous solution of equation \eqref{eq eps hj} satisfying $\tilde u^\eps_\theta(0, x) = \langle \theta, x \rangle$ for all $x\in \M$.	Assume there exist a dense subset $D$ of $\M$ and a function  $\overline H \colon D \to \R$ such that, for every $\theta \in D$, the following convergence takes place:
	\begin{equation}\label{app hom dense linear data}
		\tilde u^\eps_\theta(t, x) \ {\ucv} \ \langle \theta, x \rangle - t \overline H(\theta)
		\quad\hbox{in $\cTcyl$ as $\epsilon \to 0^+$.}
	\end{equation}
	The following holds:
\begin{itemize}
\item[\em (i)] $\overline H$ satisfies condition (H1)-(H2) on $D$  with the same $\beta>0$. In particular, it can be uniquely extended by continuity to the whole $\M$;\smallskip
\item[\em (ii)]  the convergence stated in \eqref{app hom dense linear data} holds for any $\theta\in\M$.
\end{itemize}
\end{prop}

\begin{proof}
(i) Let us prove that $\overline H$ satisfies condition (H1) on $D$. To this aim, we fix $\theta\in D$ and remark that the functions $u^+,\,u^-$ defined as 
\[
u^\pm(t,x) \coloneqq \langle \theta, x\rangle \pm \beta\left(1+ |\theta|\right)t,
\qquad
\hbox{$(t,x)\in\cTcyl$},
\]
are, respectively, a continuous super and subsolution of \eqref{eq eps hj} satisfying $u^\pm(0,x)=\langle \theta,x\rangle$, for every $\eps>0$ in view of assumption (H1). 
By Theorem~\ref{appA: teo comp2} we infer $u^-\leqslant \tilde u^\eps_\theta\leqslant u^+$ in $\cTcyl$ for every $\eps>0$. Hence 
\[
|\overline H(\theta)|
=
\lim_{\eps\to 0^+} |\tilde u^\eps_\theta(1,0)|
\leqslant 
\beta\left(1+ |\theta|\right),
\]
as it was to be shown. 
Let us now show that $\overline H$ satisfies condition (H2) on $D$. Fix $\theta_1,\theta_2\in D$ and set 
\[
u_{\theta_2}^{\eps,\pm}(t,x)
 \coloneqq 
\tilde u^\eps_{\theta_2}(t,x) +\langle \theta_1-\theta_2,x\rangle \pm\beta |\theta_2-\theta_1|t,
\qquad
(t,x)\in\cTcyl. 
\]
The function $u^{\eps,+}_{\theta_2}$ and $u^{\eps,-}_{\theta_2}$ are, respectively, a super and a subsolution of \eqref{eq eps hj}, in view of assumption (H2), which satisfy $u^{\eps,\pm}_{\theta_2}(0,x)=\langle \theta_1,x\rangle$. By Theorem~\ref{appA: teo comp2} we derive that 
$u^{\eps,-}_{\theta_2}\leqslant \tilde u^\eps_{\theta_1}\leqslant u^{\eps,+}_{\theta_2}$ in $\cTcyl$, hence 
\[
|\overline H(\theta_1)-\overline H(\theta_2)|
=
\lim_{\eps\to 0^+} |\tilde u^{\eps}_{\theta_1}(1,0)-\tilde u^{\eps}_{\theta_2}(1,0)|
\leqslant
\beta |\theta_2-\theta_1|.
\]
It is clear that such a $\overline H$ can be uniquely extended by continuity to the whole $\M$.\smallskip

{(ii)} For every $\theta\in\M$, let us set $u^\eps_\theta \coloneqq \tilde u^\eps_\theta-\langle \theta, x\rangle$. Then $u^\eps_\theta$ is the solution of \eqref{eq eps hj} with $H(\cdot,\theta+\cdot)$ in place of $H$ and initial datum $u^\eps(0,x)=0$ for all $x\in\M$. Let us fix $\theta\in\M\setminus D$ and choose a sequence $(\theta_n)_n$ in $D$ converging to $\theta$. Let us set 
\[
u^{\eps}_{n,\pm}(t,x) \coloneqq u^\eps_{\theta_n}(t,x)\pm t\beta(|\theta_n-\theta|),\qquad\hbox{$(t,x)\in\cTcyl$}.
\]
In view of assumption (H2), it is easy to check that $u^{\eps}_{n,-}$ and $u^{\eps}_{n,+}$ are, respectively, a sub and a supersolution of 
\eqref{eq eps hj} with $H(\cdot,\theta+\cdot)$ in place of $H$ and zero initial datum.  By comparison, see Theorem~\ref{appA: teo comp2}, we infer 
that $|u^\eps_{\theta}(t,x)-u^\eps_{\theta_n}(t,x)|\leqslant  t\beta(|\theta_n-\theta|)$ for all $(t,x)\in\cTcyl$, hence
\begin{eqnarray*}
\left|\tilde u^\eps_\theta(t,x) -\langle\theta,x\rangle + t\overline H(\theta)\right|
=
\left| u^\eps_\theta(t,x) + t\overline H(\theta)\right|
&\leqslant&
\left| u^\eps_{\theta_n}+ t\overline H(\theta_n)\right|
+
t |\overline H(\theta_n)-\overline H(\theta)| 
+
t\beta(|\theta_n-\theta|)\\
&\leqslant& 
\left| u^\eps_{\theta_n}(t,x)+ t\overline H(\theta_n)\right|
+ 
2T\beta(|\theta_n-\theta|)
\end{eqnarray*} 
for all $(t,x)\in\cTcyl$. The assertion follows by sending first $\eps\to 0^+$ and then $n\to +\infty$. 
\end{proof}

We will also need the following fact.

\begin{prop}\label{app prop hom}
  Let us assume that all the hypotheses of Theorem~\ref{teo reduction}
  are in force. Let $g\in\D{UC}(\R^d)$ and, for every $\eps>0$, let us
  denote by $u^\eps$ the unique function in $\CC(\cTcyl)$ that solves \eqref{eq eps hj}
  subject to the initial condition
  $u^\eps(0,\cdot)=g$. Set
\begin{eqnarray*}
u^*(t,x)& \coloneqq &\lim_{r\to 0}\ \sup\{ u^{\eps}(s,y)\,:\, (s,y)\in (t-r,t+r)\times B_r(x),\ 0<\eps<r\,\},\\
u_*(t,x)& \coloneqq &\lim_{r\to 0}\ \inf\{ u^\eps(s,y)\,:\, (s,y)\in (t-r,t+r)\times B_r(x),\ 0<\eps<r\,\}.
\end{eqnarray*}
Let us assume that $u^*$ and $u_*$ are finite valued. Then\smallskip
\begin{itemize}
\item[(i)] $u^*\in\D{USC}(\cTcyl)$ and it is a viscosity subsolution of \eqref{app effective eq};\medskip
\item[(ii)] $u_*\in\D{LSC}(\cTcyl)$ and it is a viscosity supersolution of \eqref{app effective eq}.
\end{itemize}
\end{prop}

\begin{proof}
	The fact that $u^*$ and $u_*$ are upper and lower semicontinuous on $\cTcyl$ is an immediate consequence of their definition. 
	Let us prove (i), i.e., that $u^*$ is a subsolution of \eqref{app effective eq}. 
	The proof of (ii) is analogous.

	We make use of Evans's perturbed test function method, see~\cite{Evans89}. 
	Let us assume, by contradiction, that $u^*$ is not a subsolution of \eqref{app effective eq}. Then there exists a function $\phi\in\D{C}^1(\cTcyl)$ that is a strict supertangent of $u^*$ at some point $(t_0,x_0)\in\cTcyl$ and for which the subsolution test fails, i.e.,
	\begin{equation}\label{app test fails}
		\partial_t\phi(t_0,x_0)+\overline H(D_x\phi(t_0,x_0))>2\delta
	\end{equation}
	for some $\delta>0$. For $r>0$ define $V_r \coloneqq (t_0-r,t_0+r)\times B_r(x_0)$. Choose $r_0>0$ to be small enough so that $V_{r_0}$ is compactly contained in $\cTcyl$ and $u^*-\phi$ attains a strict local maximum at $(t_0,x_0)$ in $V_{r_0}$. In particular, we have for every $r \in (0, r_0)$
	\begin{equation}\label{app strict supertangent}
		\max_{\partial V_r} (u^*-\phi)<\max_{\overline V_r} (u^*-\phi)=(u^*-\phi)(t_0,x_0).
	\end{equation}
	Let us set $\theta \coloneqq D_x\phi(t_0,x_0)$ and for every $\eps>0$ denote by $\tilde u^\eps_\theta$ the unique continuous function in $\cTcyl$ that solves \eqref{eq eps hj} subject to the initial condition $\tilde u^\eps_\theta(0,x)=\langle\theta,x\rangle$. We claim that {there is an} $r\in(0,r_0)$ such that the function
	\[
		\phi^\eps(t,x) \coloneqq \phi(t,x)+\tilde u_\theta^\eps(t,x)-\left(\langle\theta,x\rangle-t\overline H(\theta)\right)
	\]
	is a supersolution of \eqref{eq eps hj} in $V_r$ for every $\eps>0$ small enough. Indeed, by a direct computation we first get 
	\begin{eqnarray}\label{app ineq1}
		\partial_t\phi^\eps+H\left(\frac{x}{\eps},D_x\phi^\eps\right)
			=
			\partial_t\phi+\overline H(\theta)
			+ \partial_t \tilde u^\eps_\theta+H\left(\frac{x}{\eps},D_x\tilde u^\eps_\theta+D_x\phi-\theta\right)
	\end{eqnarray}
	in the viscosity sense in $V_r$. 
	Using \eqref{app test fails}, the continuity of $\overline H$ and the fact that $\phi$ is of class $C^1$, we get that there is an $r\in (0,r_0)$ such that for all sufficiently small $\eps > 0$ and all $(t,x)\in V_r$
	\[
		\partial_t\phi(t, x) + \overline{H}(\theta) > 2 \delta \,.
	\]
	Moreover, by taking into account (H2), we can further reduce $r$ if necessary to get 
	\[
		H\left(\frac{x}{\eps},D\tilde u^\eps_\theta+D_x\phi-\theta\right)
		>
		H\left(\frac{x}{\eps},D\tilde u^\eps_\theta\right)-\delta\qquad\hbox{in $V_r$}
	\]
	in the viscosity sense. Plugging these relations into \eqref{app ineq1} and using the fact that $\tilde u^\eps_\theta$ is a solution of \eqref{eq eps hj}, we finally get 
	\begin{eqnarray*}
		\partial_t\phi^\eps+H\left(\frac{x}{\eps},D_x\phi^\eps\right)
			>
			\delta+\partial_t \tilde u^\eps_\theta+H\left(\frac{x}{\eps},D\tilde u^\eps_\theta\right)
			=
			\delta
			>
			0
	\end{eqnarray*}
	in the viscosity sense in $V_r$, thus showing that $\phi^\eps$ is a supersolution of \eqref{eq eps hj} in $V_r$. 
	Now we need a comparison principle for equation \eqref{eq eps hj} in $V_r$ applied to $\phi^\eps$ and $u^\eps$ to infer that 
	\[
		\sup_{ V_r} (u^\eps-\phi^\eps){\leqslant}\max_{\partial V_r} (u^\eps-\phi^\eps).
	\]
	Since condition (H2) is in force, the validity of this comparison principle is guaranteed by \cite[Theorem 3.3 and Section 5.C]{users}. 
	Now notice that, by the assumption \eqref{app hom linear data}, $\phi^\eps\ucv\phi$ in $\overline V_r$. 
	Taking the limsup of the above inequality as $\eps\to 0^+$ we obtain
	\[
		\sup_{V_r} (u^*-\phi)
		\leqslant 
		\limsup_{\eps\to 0^+} \sup_{ V_r} (u^\eps-\phi^\eps)
		\leqslant
		\limsup_{\eps\to 0^+}
		\max_{\partial V_r} (u^\eps-\phi^\eps)
		\leqslant
		\max_{\partial V_r} (u^*-\phi),
	\]
	in contradiction with \eqref{app strict supertangent}. 
	This proves that $u^*$ is a subsolution of \eqref{app effective eq}. 
\end{proof}

We are now in position to prove Theorem~\ref{teo reduction}.

\begin{proof}[Proof of Theorem~\ref{teo reduction}]
The fact that $\overline H$ satisfies (H1)-(H2) directly follows from Proposition \ref{app prop dense} by taking $D=\M$.  
We now proceed to prove the second part of the assertion. Let us take a dense and countable subset $D \coloneqq (\theta_n)_n$ of $\M$ and set $\hat\Omega \coloneqq \bigcap_n \Omega_{\theta_n}$. Let us fix $\omega\in\hat\Omega$. According to Proposition \ref{app prop dense}, the convergence in \eqref{app hom dense linear data} holds for every $\theta\in \M$. 
We are going to show that, for any such fixed $\omega\in\hat\Omega$, the solutions 
$u^\eps(\cdot,\cdot,\omega)$ to \eqref{eq eps hj} with initial datum $u^\eps(0,\cdot,\omega)=g$ in $\M$ converge to the solution $\overline u$ of \eqref{app effective eq} with same initial datum, for any $g\in\D{UC}(\M)$. Since $\omega$ will remain fixed throughout the proof, we will omit it from our notation. 

Let us first assume $g\in {\D{C}^1(\R^d)\cap\Lip}(\R^d)$. Take a constant $M$ large enough so that 
\[
M>  \|H(x,D g(x))\|_\infty.
\]
Then the functions $u_-(t,x) \coloneqq g(x)-Mt$ and $u_+(t,x) \coloneqq g(x)+Mt$ are,
respectively, a Lipschitz continuous sub and supersolution of
\eqref{eq eps hj} for every $\eps>0$. By the Comparison Principle stated in Theorem~\ref{appA: teo comp2}, we get $u_-\leqslant u^\eps \leqslant u_+$ in $\cTcyl$ for
every $\eps>0$. By the definition of relaxed semilimits we
infer
\[u_-(t,x)\leqslant u_*(t,x)\leqslant u^*(t,x)\leqslant u_+(t,x) \quad
\text{for all $(t,x)\in\cTcyl$},\]
in particular, $u_*$, $u^*$ satisfy
$u_*(0,\cdot)=u^*(0,\cdot)=g$ on $\R^d$.  By Proposition~\ref{app prop
  hom}, we know that $u^*$ and $u_*$ are, respectively, an upper
semicontinuous subsolution and a lower semicontinuous supersolution of
the effective equation \eqref{app effective eq}. We can therefore
apply Theorem~\ref{appA: teo comp2} again 
 to obtain $u^*\leqslant u_*$ on
$\cTcyl$. Since the opposite inequality holds by the definition of upper and
lower relaxed semilimits, we conclude that the function
\[
\overline u(t,x) \coloneqq u_*(t,x)=u^*(t,x)\qquad\hbox{for all $(t,x)\in\cTcyl$}
\]
is the unique continuous viscosity solution of \eqref{app effective
  eq} such that $\overline u(0,\cdot)=g$ on
$\R^d$. Furthermore, by Theorem~\ref{teo HJ equation extended}, we also know that $\overline u$ belongs to $\UC(\cTcyl)$. The fact that the relaxed
semilimits coincide implies that  $u^\eps$ converge locally
uniformly in $\cTcyl$ to $\overline u$, see for instance \cite[Lemma
6.2, p. 80]{ABIL}.

When the initial datum $g$ is just uniformly continuous on $\cTcyl$, the
result easily follows from the above by
approximating $g$ with a sequence $(g_n)_n$ of initial data belonging to
${\D{C}^1(\R^d)\cap\Lip}(\R^d)$. Indeed, if we denote by $u^\eps_n$ and $u^\eps$ the solution of \eqref{eq eps hj} with initial datum $g_n$ and $g$, respectively, we have, in view of Theorem~\ref{appA: teo comp2}, 
\[
\|u^\eps_n-u^\eps\|_{L^{\infty}(\cTcyl)}
\leqslant 
\|g_n-g\|_{L^\infty(\M)}
\qquad
\hbox{for every $\eps>0$.}
\]
The assertion follows from this by the stability of the notion of viscosity solution. 
\end{proof}

\bibliography{references}
\bibliographystyle{siam}
\end{document}